\documentclass[12pt,reqno]{amsart}

\usepackage{latexsym, url}
\usepackage{amssymb}
\usepackage{color}
\usepackage{graphicx}

\renewcommand{\Re}{{\operatorname{Re}\,}}
\renewcommand{\Im}{{\operatorname{Im}\,}}

\renewcommand{\epsilon}{\varepsilon}

\newcommand{\R}{{\mathbb R}}
\newcommand{\C}{{\mathbb C}}

\newcommand{\Z}{{\mathbb Z}}

\renewcommand{\phi}{\varphi}

\numberwithin{equation}{subsection}

\newtheorem{theo}[equation]{{\sc Theorem}}

\newtheorem{lem}[equation]{{\sc Lemma}}

\newtheorem{prop}[equation]{{\sc Proposition}}
\newtheorem{claim}[equation]{{\sc Claim}}
\theoremstyle{definition}
\newtheorem{defn}[equation]{{\sc Definition}}

\theoremstyle{remark}
\newtheorem{rem}[equation]{Remark}

\title[Nodal length of Steklov eigenfunctions]{Nodal length of Steklov eigenfunctions on real-analytic Riemannian surfaces}
\author{Iosif Polterovich}
\address{D\'e\-par\-te\-ment de math\'ematiques et de
sta\-tistique, Univer\-sit\'e de Mont\-r\'eal CP 6128 succ
Centre-Ville, Mont\-r\'eal QC  H3C 3J7, Canada.}
\email{iossif@dms.umontreal.ca}
\author{David A. Sher}
\address{Department of Mathematics, University of Michigan, 2074 East Hall, 530 Church Street, Ann Arbor  MI  48109-1043, U.S.A.}
\curraddr{Department of Mathematical Sciences, DePaul University, 2320 N. Kenmore Ave., Chicago IL 60614,  U.S.A.}
\email{dsher@depaul.edu}
\author{John A. Toth}
\address{Department of Mathematics and
Statistics, McGill University, 805 Sherbrooke Str. West,
Montr\'eal QC H3A 2K6, Ca\-na\-da.}
\email{jtoth@math.mcgill.ca} 
 
\date{}
\begin{document}
\begin{abstract} We prove sharp upper and lower bounds for the nodal length of Steklov eigenfunctions on real-analytic Riemannian surfaces with boundary. The argument involves frequency function methods for harmonic functions in the interior of the surface as well as the construction of exponentially accurate approximations for the Steklov eigenfunctions near the boundary.
\end{abstract}

\maketitle
\section{Introduction and main results}
\subsection{Steklov problem}
Let $\Omega$ be a compact  $n$-dimensional manifold with $C^{\infty}$ boundary $\partial \Omega$ and unit exterior normal $\nu.$ We consider the Steklov eigenvalue problem
\begin{align} \label{steklov}
\Delta u_{\lambda}(x) &= 0, \,\, x \in \Omega, \nonumber \\
 \partial_{\nu} u_{\lambda}(q) &= \lambda u_{\lambda}(q), \,\, q \in \partial \Omega. \end{align}
The solutions $u_\lambda \in C^\infty (\Omega)$ are called Steklov eigenfunctions.  Let $\gamma_{\partial \Omega}: C(\Omega) \to C(\partial \Omega)$ denote the boundary restriction map. The boundary restrictions of the Steklov eigenfunctions, 
$\phi_{\lambda}:= \gamma_{\partial \Omega} u_{\lambda}$, are the eigenfunctions of the Dirichlet-to-Neumann (sometimes also referred to as the Poincar\'e--Steklov) operator $P$.  It is well-known that $P$ is a first-order homogeneous, self-adjoint, elliptic pseudodifferential operator $P \in \Psi^{1}(\partial \Omega),$  where  $P - \sqrt{-\Delta_{\partial \Omega} } \in \Psi^{0}(\partial \Omega)$ (see \cite{Ta}).  Consequently, the Steklov spectrum, which coincides with the spectrum  of $P$,  consists of an infinite sequence of eigenvalues $\lambda_j$ with $\lambda_j \to \infty$ as $j \to \infty.$

Recently, there has been a significant interest in the study of the nodal geometry  for the Steklov problem, see \cite{BLin, Ze, WZ}. However, these results 
were concerned with the nodal sets of the Dirichlet-to-Neumann eigenfunctions $\phi_\lambda$, i.e. the boundary nodal sets. The present paper focusses 
on the interior nodal sets, i.e. the nodal sets of the Steklov eigenfunctions~$u_\lambda$. With the exception of the recent preprint \cite{SWZ} (see Remark \ref{swz}), very little  has been previously known on this subject. We also note that  Courant type bounds were earlier obtained on the number of the interior (in any dimension) and boundary (in dimension two) nodal domains  of Steklov eigenfunctions (\cite{KS, AM, KKP}, see also \cite[section 6.1]{GP}). 

\subsection{Main result}  Motivated by a celebrated conjecture of Yau \cite{Yau, Yau2}, it has been recently suggested  in  \cite[Open problem 11 (i)]{GP} that the ratio between the $(n-1)$--dimensional Hausdorff measure of the nodal set of a Steklov eigenfunction $u_\lambda$  and the corresponding Steklov eigenvalue $\lambda$ is bounded above and below by some positive constants  depending only on the geometry of the manifold.  Our main result below gives a positive answer to  this conjecture  for real-analytic Riemannian surfaces. Note that, in this case,  the $1$-dimensional Hausdorff measure of a nodal set $Z_{u_\lambda}$ is simply equal
to its total length, which we will denote by $\mathcal{L}(Z_{u_\lambda})$. 
\bigskip

\begin{theo} \label{main} Let $\Omega$ be a real-analytic compact Riemannian surface with boundary.
 Then there exist  constants $C_j =C_{j}(\Omega)>0, \,  j=1,2,$ such that, for any Steklov eigenfunction $u_\lambda$ with an eigenvalue $\lambda>0$,
the total length of its nodal set $Z_{u_\lambda}$ satisfies:
$$ C_{1} \lambda \leq {\mathcal L}(Z_{u_{\lambda}}) \leq C_{2} \lambda.$$
\end{theo}
This result may be viewed as an analogue of the Donnelly-Fefferman bound for the size of the nodal set of Laplace eigenfunctions on real-analytic manifolds \cite{DF}.
\begin{rem}
To illustrate the nodal bounds in Theorem \ref{main}, consider the Steklov eigenfunctions on a unit disk corresponding to the double eigenvalue  $\lambda_{2n-1}=\lambda_{2n}=n$, $n=1,2,.\dots$.   They can be represented in polar coordinates by $u(r\,\theta)=r^n \sin(n\theta +\alpha)$ for
some $\alpha \in [0, \pi/2]$. The nodal set of each such eigenfunction is a union of $n$ diameters, of total length $2n$. 
\end{rem}

\begin{rem}
\label{swz}
Our methods here are specific to the case of real-analytic surfaces. In higher dimensions, Steklov eigenfunctions on $\Omega$ are much  more complicated and their nodal structures are not well-understood. In the case of a general smooth $n$-dimensional manifold with smooth boundary, a non-sharp lower bound
\[|Z_{\lambda}|\geq c\lambda^{\frac{2-n}{2}}\]
for the volume $|Z_{\lambda}|$ of the interior nodal sets has recently been proven in \cite{SWZ}.  It can be viewed as the Steklov analogue of the lower bounds on the size of the nodal sets of Laplace eigenfunctions obtained in \cite{CM, SZ, HS, M, HW}. 

After the first  version of the present paper was posted on the archive,  an upper bound of order $\lambda^{3/2}$ on the size of the nodal set of Steklov eigenfunctions  on surfaces 
 was obtained in \cite{Zhu} using a  quite different approach. While this bound is not sharp, it is valid for arbitrary compact smooth surfaces with boundary.
More recently, while the present paper was under review, the sharp upper bound in Theorem \ref{main} has been extended in \cite{Zhu2} to real-analytic Riemannian manifolds  of arbitrary dimension.
\end{rem}

\subsection{Sketch of the proof}
\label{sketch}
 For any closed manifold with boundary $\Omega$, Steklov eigenfunctions $u_{\lambda}$ decay rapidly into the interior of $\Omega$ \cite{HLu,GPPS}. In order to analyze their nodal lengths, we must therefore consider a neighborhood of the boundary and its complement separately. The idea is to use quasimodes near the boundary and frequency function techniques in the interior.

We begin, in section 2, by constructing quasimodes for the Dirichlet-to-Neumann operator $P$. In Proposition \ref{boundary lemma}, we show that trigonometric functions are approximate eigenfunctions for $P$, up to an exponentially decaying error term. The proof uses the $C^{\omega}$ surface assumption in a very strong way: from \cite{GPPS}, we conclude that the error term decays faster than any polynomial in $\lambda$, and the assumption of analyticity allows us to improve the decay of the error term to exponential. From Proposition \ref{boundary lemma} and some fairly standard linear algebra techniques, we show in Lemma \ref{quasimodeapprox} that the boundary Steklov eigenfunctions $\varphi_{\lambda}$ can be approximated by trigonometric functions $\psi_{\lambda}$ plus an error $f_{\lambda}$ which decays exponentially in $\lambda$ in any $C^k$ norm.

We would like to use the quasimode approximation to estimate nodal length of a Steklov eigenfunction $u_{\lambda}$ near the boundary, but from the example of the annulus (see Example \ref{annulus}), we know that it is possible for $u_{\lambda}$ to itself be exponentially small near one or more boundary components. Indeed, this seems to be the generic situation. Near these ``residual" boundary components, the error in the quasimode approximation can be larger than the quasimode itself and hence the approximation is not particularly useful. Our proof avoids this difficulty via harmonic extension of the Steklov eigenfunctions across residual boundary components. By controlling the $C^k$-norm of the extended eigenfunctions (see Lemma \ref{extension}), one can effectively treat a neighborhood of the ``residual" boundary components in the same way as the interior of $D.$

To illustrate the proof, consider a Steklov eigenfunction $u_{\lambda}$ on the annular domain $D$ in Figure~\ref{fig:mainfig}. Suppose that $\partial D_1$ is a non-residual, i.e. ``dominant", boundary component for $u_{\lambda}$ and that $\partial D_2$ is a residual boundary component. In section 3.2, we extend our boundary quasimode approximation given by Lemma \ref{quasimodeapprox} into the interior. The exponential decay of the error means that the approximation is effective in a $\lambda$-independent neighborhood of $\partial D_1$. A direct comparison of nodal sets (section 4.2) then gives us upper and lower bounds on the nodal length of $u_{\lambda}$ in a small neighborhood of $\partial D_1$, denoted by the dotted line in Figure \ref{fig:mainfig}.

\begin{figure}
	\centering
	\includegraphics[scale=1]{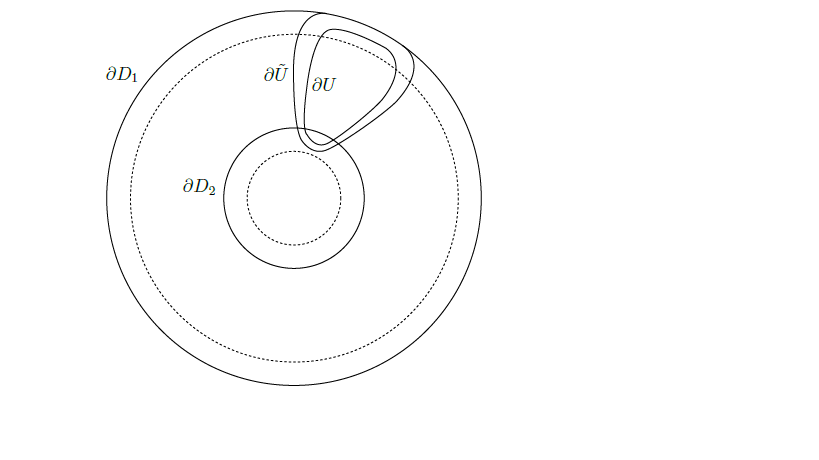}
	\caption{Illustration of the proof. The domain $D$ is the annulus between the boundary components $\partial D_1$ and $\partial D_2$, denoted by solid lines. The dotted circles mark a neighborhood of $\partial D_1$ in $D$ and a neighborhood of $\partial D_2$ in $D^c$. The two sets $U\subset\tilde U$ are as illustrated.}
	\label{fig:mainfig}
\end{figure}

To treat the interior and a neighborhood of the residual boundary components, we extend $u_{\lambda}$ (section 2.4) to a small $\lambda$-independent neighborhood of $\partial D_2$, with boundary denoted by the dotted circle. In this neighborhood, $u_{\lambda}$ is still exponentially small in $\lambda$ (section 4.3). Now let $U\subset\tilde U$ be two sets as in Figure \ref{fig:mainfig}. In section 4.4, we use standard frequency function techniques to bound the nodal length of $u_{\lambda}$ in $U$ from above by the ``renormalized Almgren frequency function" of the larger domain $\tilde U$. Then, using the quasimode approximation near the boundary and the fact that a portion of $\partial\tilde U$ coincides with $\partial D_1$, we bound this frequency function by a multiple of $\lambda$. A covering argument extends these upper bounds to all of $D$ outside a neighborhood of $\partial D_1$, including a neighborhood of $\partial D_2$. Combining these upper bounds with the two-sided bounds near $\partial D_1$ completes the proof.

\begin{rem} A particularly novel aspect of the problem is the exponential decay into the interior of the eigenfunctions, namely that for some positive constants $\tau$ and $C$ depending only on the geometry of $\Omega$,
\begin{equation}\label{decayintointerior}|u_{\lambda}(x)|\leq Ce^{-\tau d(x,\partial\Omega)\lambda}.\end{equation}
This is an immediate consequence of Lemma \ref{APPROX} and the exponential decay of the interior quasimode approximations. A similar estimate holds for the derivatives of order $k\ge 1$ of $u_{\lambda}$, provided we multiply the right-hand side by $\lambda^k$. Our problem therefore resembles the question  of estimating the size of nodal sets in forbidden regions
for eigenfunctions of Schr\"odinger operators (see \cite{HZZ, CT}). A different method of proving \eqref{decayintointerior} has been communicated to us by M. Taylor \cite{Ta1}.
\end{rem}

\begin{rem} It also follows from our results  (see Proposition \ref{cor:eigenvaluedecay})  that for real-analytic surfaces, the error term  in the eigenvalue approximation of \cite[Theorem 1.4]{GPPS} in fact decays exponentially as  the index of the eigenvalue increases.  In particular, for a simply connected real-analytic surface $\Omega$, there exists a constant $\tau >0$ depending on the geometry of $\Omega$  such that
$$
\lambda_{2j}=\lambda_{2j+1}+O(e^{-\tau j})=\frac{2\pi j}{{\rm Length}(\partial \Omega)}+O(e^{-\tau j}).
$$
This improves the error estimate $O(j^{-\infty})$  obtained independently by Rozenblyum and Guille\-min-Melrose (see \cite{Ro, Ed}) for smooth simply connected surfaces.

\end{rem}

\subsection{Example: Steklov problem on an annulus}
\label{annulus}
To illustrate several features of Steklov eigenvalues and eigenfunctions used in the proof of Theorem \ref{main} outlined in section \ref{sketch},  
consider an annulus $A(1,\epsilon)$ with inner radius $\epsilon$ and outer radius $1$. 

First we compare its Steklov spectrum to the spectrum of a union of two disks, of radii $1$ and $\epsilon$, which consists of the double eigenvalue zero, as well as the double eigenvalues $k$ and $k/\epsilon$ for each $k\in\mathbb N$. Let us show  that the difference between these eigenvalues and the  corresponding eigenvalues of $A(1,\epsilon)$ is in fact exponentially vanishing in $k$.  Indeed, as computed in \cite{D} (see also \cite[Section 4.2]{GP}), the nonzero Steklov eigenvalues of the annulus are given by the roots of the quadratic polynomial
\[p_k(\sigma)=\sigma^2 - \sigma k\left(\frac{\epsilon+1}{\epsilon}\right)\left(\frac{1+\epsilon^{2k}}{1-\epsilon^{2k}}\right)+\frac{1}{\epsilon}k^2,\]
for each $k\in\mathbb N$ (each root corresponds to a double eigenvalue). 
We could compute the eigenvalues directly, but it is easier to compare $p_k(\sigma)$ to the polynomial
\[\tilde p_k(\sigma)=\sigma^2 - \sigma k\left(\frac{\epsilon+1}{\epsilon}\right)+\frac{1}{\epsilon}k^2,\]
which has roots $k$ and $k/\epsilon$. In fact $\tilde p_k(\sigma) - f(k,\epsilon) \sigma =p_k(\sigma)$, where 
\[f(k,\epsilon)=k\left(\frac{\epsilon+1}{\epsilon}\right)\frac{2\epsilon^{2k}}{1-\epsilon^{2k}}.\] Since $f(k,\epsilon)=O(k\epsilon^{2k})$, a straightforward calculation shows that the roots of $p_k(\sigma)$ differ from $k$ and $k/\epsilon$ by $O(k\epsilon^{2k})$ as well. The error is exponentially decreasing in $k$, which agrees with the eigenvalue approximation result in Proposition \ref{cor:eigenvaluedecay}.

Steklov eigenfunctions on surfaces also exhibit exponential decay at certain boundary components, which motivates our definition of ``dominant" and ``residual" boundary components (see Definition \ref{dominant}). To illustrate this, consider an eigenfunction $u_{\sigma_j}$ of the annulus $A(1,\epsilon)$ corresponding to an eigenvalue $\sigma_j=k+O(k\epsilon^{2k})$. As computed in \cite{GP},
\[u_{\sigma_j}(r,\theta)=C_k\left(r^{k}+\frac{k-\sigma_j}{k+\sigma_j}r^{-k}\right)T(k\theta),\]
where $T(k\theta)$ is a unit norm linear combination of $\cos(k\theta)$ and $\sin(k\theta)$ and $C_k$ is an appropriate normalizing constant. Since $k-\sigma_j=O(k\epsilon^{2k})$, we see that $C_k\sim 1$, that the $L^2$-norm of $u_{\sigma_j}$ on $\{r=1\}$ is roughly $1$, and that the $L^2$-norm on $\{r=\epsilon\}$ is $O(\epsilon^k)$ - which is exponentially decaying in $k$. So, in this case,  the circle of radius $1$ is a dominant boundary component, while the circle of radius $\epsilon$ is a residual boundary component. A similar analysis holds for the eigenfunctions corresponding to eigenvalues of the form $k/\epsilon + O(k\epsilon^{2k})$, with the roles of the boundary components reversed.


\subsection*{Acknowledgments} The authors are grateful to M. Sodin and M. Taylor for useful discussions, and to the anonymous referee for many helpful suggestions on the presentation of the paper.  The research of I.P. was partially supported by NSERC, FRQNT and Canada Research Chairs Program. The research of D.S was partially supported by NSF EMSW21-RTG 1045119. The research of J.T. was partially supported by NSERC and FRQNT. I.P. and J.T. were also supported by the French National Research Agency project Gerasic-ANR-13-BS01-0007-0.

\section{ Steklov problem for real-analytic planar domains}
\subsection{Koebe uniformization}
We first prove Theorem \ref{main} in the case where $\Omega$ is a real-analytic planar domain and explain the fairly minor modifications needed to treat the general case of compact Riemann surfaces with boundary  in  section \ref{general}. The two-dimensional case is quite special. Indeed, unlike the higher-dimensional case, for surfaces the Steklov operator $P: C^{\infty}(\partial \Omega) \to C^{\infty}(\partial \Omega)$ agrees with the square root of the Laplacian on the various boundary components modulo  a smoothing operator. Upon reparametrization by arclength, the latter operator is consequently (modulo smoothing) 
just the Fourier multiplier $M$ acting on the component boundary circles. Although not necessary, one can see the former directly by applying conformal mapping.
 Indeed, by the Koebe uniformization theorem \cite{HSch}, one can conformally map $\Omega$ to a planar domain $D$, where $D$ is the disk of radius 1 with a finite number of interior disks removed. In the cases where $\Omega$ is simply-connected, this reduces to the Riemann mapping theorem. When $\partial \Omega$ is real-analytic, it follows by $C^{\omega}$ regularity up to the boundary in the associated Dirichlet problem for the Green's function \cite{MN} that   the conformal map $f: D\to\Omega$ extends to a real-analytic map of $\overline{D}$ to $\overline{\Omega}.$ Moreover, $f$ maps boundary to boundary in a univalent fashion. In particular, the induced boundary restriction $f_{\partial \Omega}: \partial D \to \partial \Omega$ is a $C^{\omega}$-diffeomorphism.

The boundary $\partial D$ consists of a union of circles, which we denote by  $\partial D_1$, $\partial D_2,\dots,\partial D_k$, with radii $\rho_1,\dots,\rho_k$ and centers $c_1,\dots,c_k$ respectively. The corresponding boundary components of $\Omega$ are denoted by $\partial \Omega_j := f_{\partial \Omega}(\partial D_j); j=1,2,...,k.$  We let $\theta_j\in [0,2\pi]$ be the usual angle coordinate on $\partial D_j$ for each $j$, and let the arc length coordinate $q_j$ be $\rho_j\theta_j$. Let $q$ be an arc length coordinate on $\partial D$ which coincides with $q_j$ when we restrict attention to $\partial D_j$. Finally, let $dq$ and $dq_j$ be the measures on $\partial D$ and $\partial D_j$ respectively induced by the Euclidean measure $dx$ on $\mathbb R^2$. As in \cite{Ed}, the Steklov problem in (\ref{steklov}) is conformally mapped to the problem
\begin{align} \label{disc}
\Delta v_{\lambda}(x) &= 0, \,\,\, x\in D, \nonumber \\
\partial_{\nu}v_{\lambda} (q) &= \lambda g v_{\lambda}(q),\,\, q \in \partial D, 
\end{align}
with $g=df(q)=|f'(q)|$ analytic for $ q \in \partial D$ and $g\neq 0$. By conformal mapping, $v_{\lambda}=f^{*}u_{\lambda}$ - so without loss of generality, it suffices to work with the conformal model (\ref{disc}) and we will do so here. We will abuse notation somewhat and denote the corresponding Steklov operator in the conformal model (\ref{disc}) by $P: C^{\infty}(\partial D) \to C^{\infty}(\partial D)$ as well, and its eigenfunctions by $\phi_{\lambda}$. We henceforth make the normalization that
\begin{equation} \label{normalization}
\| \phi_{\lambda} \|_{L^2(\partial D)} =1. \end{equation}
Since we consistently work with the model case in (\ref{disc}), this should not lead to confusion.

Given $\tau=(\tau_1,\dots,\tau_k) \in (0,1)^k,$ we complexify $\partial D$ to a Grauert tube  (ie. an annulus)
$$\partial D^{\C}_{\tau} = \{ \oplus_{j=1}^k e^{i\theta_j -\rho_j}; (\rho_j,\theta_j) \in [0,\tau_j) \times [0,2\pi] \}.$$
We choose $\tau_j >0$ here to be the {\em analytic modulus} of $g|_{\partial D_j} \in C^{\omega}(\partial D_j)$; that is, the maximal tube radius for which
$g$ has a holomorphic extension to $(\partial D_j)^{\C}_{\tau_j}$. 

We also note that the length $L_j$ of $\partial\Omega_j = f_{\partial \Omega}(\partial D_j)$ is $\int_{\partial D_j}g(q_j)\ dq_j$, and define a new coordinate on $\partial D_j$ by $s_j(q_j) = \int_{0}^{q_j} g(q_j')\ dq_j' \in [0,L_j]$. Let $s$ be the corresponding coordinate on all of $\partial D$. Note that since $g$ is analytic and strictly positive, this is an analytic reparametrization with analytic inverse.

\subsection{Potential layer formulas and the Steklov operator} \label{potential theory}
We briefly review the characterization of the Steklov operator $P: C^{\infty}(\partial D) \to C^{\infty}(\partial D)$ in terms of potential layer operators.  This material here is well-known and further details can be found in \cite[Sec. 7.1]{Ta}.   
Here, we assume that $\Omega \subset \R^n$ is a bounded domain with $C^{\infty}$-boundary. Let $G(x,x') \in {\mathcal D}'(\R^n \times \R^n)$ be the ambient free Green's function for $-\Delta = - \sum_{j=1}^{n} \partial^2_{ x_j}$ in $\R^n$.
Consider the single and double layer operators ${\mathcal Sl}: C^{\infty}(\partial \Omega) \to C^{\infty}(\Omega \cup (\R^n \setminus \overline{\Omega} ) )$ and ${\mathcal Dl}: C^{\infty}(\partial \Omega) \to C^{\infty}(\Omega  \cup (\R^n \setminus \overline{\Omega} )  )$ given by
\begin{align} \label{layer}
 {\mathcal Sl} f (x)  &= - \int_{\partial \Omega} G(x,q) f(q) d S(q), \nonumber \\
 {\mathcal Dl} f (x) &= - \int_{\partial \Omega} \partial_{\nu(q)} G(x,q) f(q) d S(q), \,\,\, x \in \Omega \cup (\R^n \setminus \overline{\Omega} ) . \end{align}
 Corresponding to $ {\mathcal Sl}$ and ${\mathcal Dl}$ are the boundary layer operators $S$ and $N:C^{\infty}(\partial\Omega)\to C^{\infty}(\partial\Omega)$ given by
 \begin{align} \label{layer2}
 S f (q)  &= - \int_{\partial \Omega} G(q,q') f(q') d S(q'), \nonumber \\
 N f (q) &=  - 2 \int_{\partial \Omega} \partial_{\nu(q')} G(q,q') f(q') d S(q'), \,\,\, q \in  \partial \Omega. \end{align} 
These operators are classical  pseudodifferential with $S, N \in \Psi^{-1}(\partial \Omega)$, and $S$ is elliptic. 
 Given a function $v  \in C^0(\R^n\setminus\partial \Omega)$ and $q \in \partial \Omega,$ let $v_+(q)$ and $v_{-}(q)$ denote the limits of $v(x)$ as $x \to q$ from $x \in \Omega$ and $x \in \R^n\setminus \overline{\Omega}$ respectively. The layer potential operators in (\ref{layer}) and the induced boundary operators in (\ref{layer2}) are linked via the boundary jumps equations
 \begin{align} \label{bdyjumps}
 {\mathcal Sl} f_{+} (q) &= {\mathcal Sl} f_{-}(q) = Sf(q),  \nonumber \\
 {\mathcal Dl} f _{\pm}(q) &= \pm \frac{1}{2} f(q) + \frac{1}{2} Nf(q), \,\,\,\, q \in \partial \Omega. \end{align}

 Now let $P$ be the Dirichlet-to-Neumann operator for $\Omega$. Consider the Dirichlet problem
 $$ \Delta u = 0 \,\, \text{on} \,\, \Omega, \,\,\,\, u|_{\partial \Omega} = f.$$
 A straightforward application of Green's formula gives
 \begin{equation} \label{intgreen}
 u(x) =  {\mathcal Dl} f(x) - {\mathcal Sl} \, P f (x),\,\,\, x \in \Omega. \end{equation}
 Taking limits in (\ref{intgreen}) from within $\Omega$ and applying (\ref{bdyjumps}) gives
\begin{equation} \label{key}
SP = -\frac{1}{2}( I - N). \end{equation}
Since $S \in \Psi^{-1}(\partial \Omega)$ and is elliptic, it follows from (\ref{key}) by a parametrix construction in the standard pseudodifferential calculus that $P \in \Psi^{1}(\partial \Omega)$ and is also elliptic.  

In the case where $\partial \Omega$ is real-analytic, $P$ is also analytic pseudodifferential in the sense of \cite{BK} (see also \cite[Ch. 5]{Tr}), and we write $P\in \Psi^{1}_{a}(\partial \Omega).$   To see this, we first note integration against $G(x,x')$ in $\mathbb R^n$ is a pseudodifferential operator $\mathcal I$ with full symbol $|\xi|^{-2}$ in the usual coordinates. As the subprincipal symbols are all zero, the symbol satisfies Cauchy estimates and so, $\mathcal I\in\Psi^{-2}_a(\mathbb R^n)$. Since $\partial\Omega$ is real-analytic, the Fermi coordinates $(\bar x,x_n)$ near $\partial\Omega$ are also real-analytic. In terms of these local coordinates, $\sigma_{\mathcal I}(\bar x,x_n,\bar\xi,\xi_n)$ is thus an analytic symbol; hence so is $\partial_{x_n}\sigma_{\mathcal I}(\bar x,x_n,\bar\xi,\xi_n)$. It is easy to check that the symbols of $S$ and $N$ are given, respectively, by 
 $\int_{\mathbb R} \left(\sigma_{\mathcal I}|_{x_n=0} \right) d\xi_n$ and $2\int_{\mathbb R} \left(\left(\partial_{x_n}\sigma_{\mathcal I} + i\xi_n \sigma_{\mathcal I}\right)|_{x_n=0}\right)d\xi_n$ respectively, and hence are also analytic. So $S$, $N \in \Psi^{-1}_a(\partial\Omega).$ By constructing a parametrix for $S$ in the analytic pseudodifferential calculus  (see \cite[Ch. 5]{Tr} for details), then multiplying (\ref{key}) by this parametrix on the left and rearranging, it follows that 
$$P \in \Psi^{1}_{a}(\partial \Omega).$$
In the following, we say that an operator $R \in \Psi^{*}_{a}(\partial \Omega)$ is {\em analytic smoothing} if its Schwartz kernel $K(\cdot, \cdot) \in C^{\omega}(\partial \Omega \times \partial \Omega)$, and we write $R \in \Psi^{-\infty}_{a}(\partial \Omega).$

\begin{rem} In what follows, we use the following notation: given a set $X$ and two non-negative functions $f: X \to \R$ and $g: X \to \R$, throughout the paper  we write $f \lessapprox g$ if there exists a constant $C>0$ such that $f(x) \leq C g(x)$ for all $x \in X.$  In addition, $f \approx g$ will mean that  both $f \lessapprox g$ and $g \lessapprox f$ are satisfied. 
\end{rem}
\subsection{Quasimodes for the Steklov operator}
\label{boundvalues}
It is well-known (\cite{GPPS}, see also \cite{Ed, HL})  that for smooth bounded planar domains,
\begin{equation} \label{smooth}
 P = M +R, \,\, R \in \Psi^{-\infty}(\partial D),\end{equation}
 where $M=\oplus_{j=1}^k M_j$, with $M_j:C^{\infty}(\partial D_j) \to C^{\infty}(\partial D_j)$ the Fourier multiplier defined by 
 $$ M_j e^{2\pi ims_j/L_j} = \frac{2\pi|m|}{L_j} e^{2\pi ims_j/L_j}, \,\, m \in \Z. $$
   Note that $M=\sqrt{-\Delta_{\partial D}}$. As in \cite{GPPS}, let $\mathcal{A}(\alpha)$ be the sequence $\{0,\alpha,\alpha,2\alpha,2\alpha,\dots\}$ and $\mathcal{A}(\alpha_1,\dots,\alpha_k)$ be the union of these sequences, with multiplicity, arranged in the appropriate order. We write
 \begin{equation} \label{qmspec}
  \mu_n := \left\{\mathcal{A}\left(\frac{2\pi}{L_1},\dots,\frac{2\pi}{L_k}\right)\right\}_n, \end{equation}
  which is the $n$th eigenvalue of the operator $M$. As a consequence of \eqref{smooth} and \cite{GPPS}, the eigenvalues of $P$ consist of the sequence
$$\lambda_n = \mu_n + O(n^{-\infty}), \,\, n \in \Z^+.$$
It follows that $\lambda_n \approx n$ as $n \to \infty.$
We abuse notation somewhat and denote the corresponding  orthonormal basis of eigenfunctions of $P$ on $\partial \Omega$ by $\{\phi_{\lambda_n}\}$, so that each $\phi_{\lambda_n}$ is the restriction of $u_{\lambda_n}$ to $\partial\Omega$.
Although the operator $M$ acting on $\partial D$ is non-local, it clearly maps each of the boundary components $\partial D_j$ to itself; that is, for any $f \in C^{\infty}(\partial D)$ with supp $f \subset \partial D_j,$
$$ M f  = M_j f.$$
Denote the $L^2$-normalized eigenfunctions of $M$ on the boundary component $\partial D_j$  by $e_{m,j}(s_j),$ where
\begin{align} \label{qm}
e_{m,j}(s_j) &= e^{2\pi ims_j/L_j},   \,\, s_j \in [0,L_j]. \end{align}
We also let $e_{m,j}(s)$ be the function which is $e_{m,j}(s_j)$ on $\partial D_j$ and 0 on the other boundary components, with an analogous definition for $\varphi_{\lambda_n,j}$.

The following lemma is central to the proof of Theorem \ref{main} and shows that when $\Omega$ is analytic, the functions $e_{m,j}$ are quasimodes for $P$  to exponential error in $m.$ 
\begin{prop} \label{boundary lemma}
Suppose that $\Omega$ is analytic, with $D$, the conformal map $f:D\to\Omega$, and the Dirichlet-to-Neumann operator $P: C^{\infty}(\partial D) \to C^{\infty}(\partial D)$ as in \eqref{disc}. Then:
\begin{enumerate} \item The remainder $R:C^{\infty}(\partial D) \to C^{\infty}(\partial D)$ in (\ref{smooth}) has Schwartz kernel $R \in C^{\omega}(\partial D \times \partial D)$ with analytic modulus at least $\tilde{\tau} >0$ in each variable separately. Here, $\tilde{\tau}$ depends on the analytic modulus of the conformal multiplier $g \in C^{\omega}$ and the geometry of $D$.
\item As $|m|\to\infty$,
$$\,\,  \| P e_{m,j}(s) - \frac{2\pi|m|}{L_j}  e_{m,j}(s) \|_{C^k(\partial D)}  = O_k(e^{- 2\pi \tilde{\tau} |m|/L_j}).$$
In the last line, $O_k(e^{- 2\pi \tilde{\tau} |m|/L_j})$ denotes a non-negative term that is
$\leq C_k e^{- 2\pi \tilde{\tau} |m|/L_j}$ with constant  $C_k>0$ depending on $k.$
\end{enumerate}
\end{prop}

\begin{proof}
Our proof uses the boundary integral equation for $\phi_{\lambda}$. For $x$, $x'\in\mathbb R^2$, we continue to let $G(x,x')$ be the free Green's kernel for $-\Delta_{\R^2} = - \partial_x^2 - \partial_{x'}^2$, which is $-\frac{1}{2\pi}\log|x-x'|$. Using \eqref{key} (see also \cite{Sh,Ro}), the functions $\phi_{\lambda}$ satisfy the boundary jumps equation:

 \begin{equation}\label{bdyjumps2}
 \frac{1}{2} \phi_{\lambda_n}(q)  = \lambda_n \int_{\partial D } G(q,q')  \phi_{\lambda_n}(q') \, g(q')\, dq'
 - \int_{\partial D } K(q,q') \phi_{\lambda}(q') \, dq',  \end{equation}
with
$$
K(q,q')=\partial_{\nu(x')}G(x,x')|_{(x,x')=(q,q')}=-\frac{1}{4\pi}\frac{\partial_{\nu(x')}|x-x'|^2}{|x-x'|^2}|_{(x,x')=(q,q')}.
$$
Change variables in the first integral to integrate in terms of $s'$:
\begin{equation}\label{bdyjumps3}
 \frac{1}{2} \phi_{\lambda_n}(q(s))  = \lambda_n \int_{\partial D } G(q(s),q(s'))  \phi_{\lambda_n}(q(s')) \, ds'
 - \int_{\partial D } K(q(s),q(s')) \phi_{\lambda}(q(s')) \, dq(s').
\end{equation}

We now claim:

\begin{claim}\label{decompint}
\[\frac 12\varphi_{\lambda_n}=\lambda_n \tilde{G} \varphi_{\lambda_n}-\lambda_nK_1\varphi_{\lambda_n}-K_2\varphi_{\lambda_n},\]
where $K_1$, $K_2\in\Psi^{-\infty}_a(\partial D)$ and $\tilde{G} = \oplus_{j=1}^{k} G_j$ with $G_j:C^{\infty}(\partial D_j)\to C^{\infty}(\partial D_j)$ given by
$$(G_j h)(s) = -\frac{1}{2\pi} \int_{0}^{L_j} \log |e^{2\pi is/L_j}- e^{2\pi is'/ L_j}|  \, h(s') ds'.$$
\end{claim}
\noindent Recall that $\Psi^{-\infty}_a(\partial D)$ is the space of operators with $C^{\omega}$ Schwartz kernels. We also claim:
\begin{claim}\label{logint} If we extend $M_j^{-1}$ to act on $C^{\infty}(\partial D_j)$ by letting it be zero when acting on constants, 
$$ M_j^{-1} = 2G_j.$$
\end{claim}

Assuming both claims, we now prove Proposition \ref{boundary lemma}. The claims combine to show
\[\varphi_{\lambda_n}=\lambda_nM^{-1} \varphi_{\lambda_n}-2\lambda_nK_1\varphi_{\lambda_n}-2K_2\varphi_{\lambda_n}=M^{-1}P\varphi_{\lambda_n}
- 2\lambda_nK_1\varphi_{\lambda_n}- 2K_2\varphi_{\lambda_n}.\] Given $f \in C^{\infty}(\partial D)$, we can write $f = \sum_{k=1}^{\infty} \alpha_k \phi_{\lambda_k}$ with $\alpha_k = O(k^{-\infty})$. By linearity, the boundedness of $K_1$ and $K_2$, and the rapid decay of $\alpha_k$ and of $\alpha_k\lambda_k$, it follows that
\begin{equation}
M^{-1}Pf-f = 2K_1Pf + 2K_2f. \end{equation}
Since $K_{1},K_2 \in \Psi^{-\infty}_a(\partial D)$ and $P \in \Psi^1_a(\partial D),$ standard theory of analytic pseudodifferential operators \cite[Ch. 5]{Tr} implies that $\tilde{K}=2K_1P+2K_2\in\Psi^{-\infty}_a(\partial D)$, with
\begin{equation} \label{decomp upshot}
M^{-1}Pf-f =  \tilde{K} f,\ \textrm{or equivalently }(P-M)f=M \tilde{K} f. \end{equation}

Now, we choose $f(s)= e_{m,j}(s)$. For any $\tilde{\tau}<\tau_j$, by contour deformation we have
\begin{multline} \tilde{K}e_{m,j}(s) = \int_{\partial D_j} \tilde{K}(s,s_j) e^{2\pi ims_j/L_j} ds_j\\ =  e^{-2\pi|m| \tilde{\tau}/L_j} \Big( \int_{\partial D_j} \tilde{K}(s, s_j \pm i\tilde{\tau}) e^{2\pi ims_j/L_j} ds_j \Big) = O(e^{-2\pi|m| \tilde{\tau}/L_j }), \end{multline}
where the choice of $\pm$ is determined by the sign of $m$. The same is true for each $(\partial^{k}_t \tilde{K}e_m)(s)$, so for all $k$,
\[|| \tilde{K}e_{m,j}||_{C^k}=O_k(e^{-2\pi|m|\tilde{\tau}/L_j}).\]
The same is therefore true for $|| \tilde{K}e_{m,j}||_{H^k}$ for any $k$. Since $M\in\Psi^1(\partial D)$ is a (classical) pseudodifferential operator, it is bounded from $H^k\to H^{k-1}$ for each $k$. From this and Sobolev embedding we conclude that $||M\tilde{K}e_{m,j}||_{C^k}=O_k(e^{-2\pi|m|\tau_0/L_j})$ for any $k$. Combining this with \eqref{decomp upshot} completes the proof of the proposition.
\end{proof}


We now prove both claims.
\begin{proof}[Proof of Claim \ref{decompint}] Assume that $q(s)\in\partial D_j$. The portions of the integrals in \eqref{bdyjumps3} over $\partial D\setminus\partial D_j$ can be absorbed into $K_1$ and $K_2$, since $d(\partial D_i,\partial D_j)>0$ for $i\neq j$ and therefore the integrands are real-analytic functions of $s'$. Using the symbol $\cong$ to denote equivalence up to terms of the form $(\lambda_n K_1+K_2)\varphi_{\lambda_n}$ with $K_1$ and $K_2\in\Psi^{-\infty}_a(\partial D)$, we have

\begin{multline}\label{bdyjumps4}
 \frac{1}{2} \phi_{\lambda_n}(q_j(s_j))  \cong  \lambda_n \int_{0}^{L_j} G(q_j(s_j),q_j(s_j'))  \phi_{\lambda_n}(q_j(s_j')) \, ds_j'\\
 - \int_{\partial D_j } K(q_j(s_j),q_j(s_j')) \phi_{\lambda}(q_j(s_j')) \, dq_j(s_j').
\end{multline}

For notational simplicity, we now suppress all $j$ subscripts. The last term in \eqref{bdyjumps4} can be absorbed, because the normal derivative $K(q(s),q(s'))$ is itself real-analytic in $(s,s') \in [0,L] \times [0,L].$ To see this, observe that analyticity away from the diagonal $s = s'$ is clear. To establish the analyticity near $s=s'$ we insert the Taylor expansion $q(s) = q(s') + (s-s') d_{s'}q(s') + \frac{\kappa(s')}{2} |q(s)-q(s')|^2 \nu(s') + [a_T(s,s') d_{s'}q(s') + a_N(s,s') \nu(s')] $ into the formula for $K(q(s),q(s'))$ in (\ref{bdyjumps2}). Here $\kappa(s)$ denotes the curvature at the point $s$. 
The result is that $$ K(q(s),q(s') ) = -\frac{1}{4 \pi} \Big( \frac{\kappa(s')}{2} +  \frac{a_N(s,s')}{|q(s)-q(s')|^2} \Big),$$
and so, $(s,s') \mapsto K(q(s),q(s')) \in C^{\omega}([0,L] \times [0,L])$ since $a_N$ is real-analytic with $a_N(s,s') = O(|s-s'|^3)$ as $s\to s'$. 

Finally, all we need to do to complete the proof is to show that 
\[\frac{1}{2\pi}\log |q(s)-q(s')|-\frac{1}{2\pi}\log |e^{2\pi i s/L}-e^{2\pi i s'/L}|= \frac{1}{2\pi} \log\frac{|q(s)-q(s')|}{|e^{2\pi i s/L}-e^{2\pi i s'/L}|}\]
is analytic. However, the reparametrization $q\to s(q)$ is analytic with analytic inverse, and Taylor expansion of the quotient on the right-hand side shows that it is analytic and nonzero at $s=s'$ (cf.  (\cite[formula (8.8)]{TZ})). Therefore, its logarithm is also analytic there. Analyticity away from $s=s'$ is automatic. \end{proof}

\begin{proof}[Proof of Claim \ref{logint}] By scaling, it suffices to prove the claim for $j=1$ and $L_1=2\pi$. 
We will show \begin{equation} \label{log2}
 -\frac{1}{\pi} \int_{0}^{2\pi} \log |e^{is} -e^{is'}|  e^{ims'} ds' = \frac{1}{|m|} e^{ims}, \,\,\,m \in \Z \setminus{\{0\}}. \end{equation}
The claim follows since $\{e^{ims}\}_{n\in\mathbb Z}$ is a Fourier basis. To prove \eqref{log2}, consider the harmonic function on the disk whose Neumann data is $e^{ims}$. On the one hand, this is $\frac{r^{|m|} }{|m|}e^{ims}$. On the other hand, the Poisson kernel for the Neumann problem is $-\frac{1}{\pi}\log |x-x'|$, so
\[\frac{r^{|m|} }{|m|}e^{ims}=\int_0^{2\pi}-\frac{1}{\pi}\log|re^{is}-e^{is'}|e^{ims'}\ ds'.\]
Restricting to $r=1$ shows \eqref{log2}.\end{proof}

 \subsection{Harmonic extension of interior eigenfunctions across the boundary}
It is well known that a harmonic function on a domain can be extended harmonically across a real analytic boundary (see \cite{LM}). The following lemma, which will be used in the proof of Theorem \ref{main}, provides an explicit estimate on the $C^k$-norms of the harmonic extension of a Steklov eigenfunction $u_\lambda$ across a boundary component $D_j$ in terms of the  $L^2$-norm of its trace $\phi_\lambda$ on $\partial D_j$.

\begin{prop}\label{cont prop} Let  $u_{\lambda} \in C^{\omega}(D)$ be a Steklov eigenfunction. Then, for every connected component $\partial D_j$ of the boundary, and  $0< \epsilon < \tau$  with $\tau$ defined in (\ref{tau}), there is a harmonic continuation of $u_{\lambda}$ (also denoted $u_{\lambda}$) across the boundary component $\partial D_j$ to an annulus $A_j^+(\epsilon)$ of width $\epsilon$. Moreover, for $k=0,1,2,...$ one has  the exterior estimate
$$ \| u_{\lambda} \|_{C^k(A_j^+(\epsilon))} \leq C_{k,\epsilon} \, \lambda^{k+1} e^{\lambda \epsilon}  \, \| \phi_{\lambda} \|_{L^2(\partial D_j)}. $$
\end{prop}

\begin{proof}
 From  Proposition  $\ref{boundary lemma},$ we know $P = M + R$ where $R \in \Psi^{-\infty}_a(\partial D)$. We use this characterization of $P$ to estimate Fourier coefficients of $\phi_{\lambda} |_{\partial D_j}$ along the boundary circle $\partial D_j.$ To simplify the writing in the following we assume  that $L_j =2\pi$ and that the center of the disc $D_j$ is the origin.  Also, we abuse notation somewhat and simply write $\phi_{\lambda}(s)$ for $\phi_{\lambda}(q_j(s))$ in (\ref{fc})-(\ref{fc2}) below.
 
 Then, for $k \in \Z\setminus\{0\},$
\begin{align} \label{fc}
\hat{\phi_{\lambda}}(k) &= \langle e^{isk}, \phi_{\lambda} \rangle
 = \frac{1}{|k|} \langle M e^{isk}, \phi_{\lambda} \rangle = \frac{1}{|k|} \langle e^{isk}, M \phi_{\lambda} \rangle = \frac{1}{|k|} \langle e^{isk},  (P -R) \phi_{\lambda} \rangle \nonumber \\
&= \frac{\lambda}{|k|} \langle e^{isk},  \phi_{\lambda} \rangle   + \frac{1}{|k|} O(e^{-\tau |k|}) \| \phi_{\lambda} \|  = \frac{\lambda}{|k|} \hat{\phi_{\lambda}}(k) + \frac{1}{|k|} O(e^{-\tau |k|}) \| \phi_{\lambda} \|.
\end{align}
From (\ref{fc}), we observe
\begin{equation} \label{fc1} (|k|-\lambda)\hat{\phi_{\lambda}}(k)=O(e^{-\tau |k|}),\end{equation}
and hence,  
\begin{equation}\label{fc2} \hat{\phi_{\lambda}}(k)=O(e^{-\tau |k|}), \,\,\,\, |k| \geq \lambda + 1.\end{equation}

It follows that $\phi_{\lambda}$ holomorphically continues to the strip $[0,2\pi] + i [-\epsilon, \epsilon]$ and so $\phi_{\lambda} |_{\partial D_j}$ holomorphically continues to an annular neighbourhood of $\partial D_j$ of any width $\epsilon <  \tau.$ In terms of the parametrizing coordinates $s \in [0,2\pi],$ one holomorphically continues $\phi_{\lambda}$ to the annulus
$$A_j(\epsilon):= \{ x \in \C;  x = e^{i w},  \, w \in  [0,2\pi] + i[-\epsilon,\epsilon] \}.$$
 Without loss of generality, we assume here that the set where $\Im w <0$ (denoted by $A_j^-(\epsilon)$) is the part of the annulus lying inside the domain $D$ and $\Im w >0$ (denoted by $A_j^+(\epsilon)$)  is the part lying outside. Note that $\epsilon>0$ is {\em independent} of $\lambda.$ It also follows from (\ref{fc1}) that the holomorphic continuation, $\phi_{\lambda}^{\C}$  of the boundary Steklov eigenfunction $\phi_{\lambda}$ satisfies
\begin{equation} \label{hol est}
|\phi_{\lambda}^{\C}(x)| \leq C_{\epsilon} e^{ \lambda |\Im w|}  \| \phi_{\lambda} \|_{L^2(\partial D_j)}, \quad  x = e^{iw}, \end{equation}
with an appropriate constant $C_{\epsilon}>0.$

Now we need to construct the harmonic continuation of the interior Steklov eigenfunction $u_{\lambda} \in C^{\omega}(D).$ By Green's formula,
\begin{equation} \label{cont1}
u_{\lambda}(x)  = ({\mathcal Dl} \phi_{\lambda})(x) - \lambda({\mathcal Sl} g \phi_{\lambda})(x), \,\,\, x \in D. \end{equation}
Since the harmonic continuation of both terms in (\ref{cont1}) is carried out in the same way,  we consider here the second term and then just indicate the minor changes necessary to deal with the first one.  Since the multiplicative factor $\lambda$ is irrelevant to the continuation, we just consider
\begin{equation} \label{cont2}
  ({\mathcal Sl} g \phi_{\lambda})(x) = \frac{1}{2\pi} \int_{\partial D_j} \log | x - q|  \, g(q) \, \phi_{\lambda}(q) \, dq + R \phi_{\lambda} (x), \,\,\, x \in A_j^-(\delta)\end{equation}
  where $R \phi_{\lambda} =\frac{1}{2\pi} \int_{\partial D\setminus\partial D_j}  \log| x-q| \,  g(q)  \,\phi_{\lambda}(q) dq$. That $R\phi_{\lambda}(q)$ harmonically continues to $x \in A_{j}^{+}(\epsilon)$ is immediate since the kernels $G(x,q)$ with $(x,q) \in A_j(\epsilon) \times \partial D_k, j \neq k$ are supported off-diagonal and so $G(x,q)$ is harmonic in each variable. Thus, it is enough to analyze the first term on the RHS of (\ref{cont2}).

 
 The arclength parametrization of $\partial D_j$  is given by $q(s) = e^{is}  \in \partial D_j, \,\, s \in [0,2\pi]$. We let  $\Theta_{\lambda}(s):=  \int_{0}^{s} g(q(s')) \phi_{\lambda}(q(s')) dq(s').$   First, by application of Green's formula in (\ref{disc}) with harmonic measure $ \delta_{\partial D_j}$, it follows that $\int_{0}^{2\pi} g(q(s')) \phi_{\lambda}(q(s')) dq(s') = 0$ and so, in particular, $\Theta_{\lambda} \in C^{\omega}([0,2\pi]).$ Moreover, since $$| \widehat{\Theta_{\lambda}(k)}| \lessapprox |k|^{-1} |\widehat{ g  \phi}_{\lambda}(k)|,$$
  it  has the same holomorphic continuation properties as $ g \phi_{\lambda} (s)$ (with continuation to the strip $[0,2\pi] + i[-\epsilon,\epsilon]).$  Then, making the definition
  $$\Phi_{\lambda}(e^{is}): = \Theta_{\lambda} (s), \quad s \in [0,2\pi],$$
it follows   that  $\Phi_{\lambda} \in C^{\omega}(\partial D_j)$ holomorphically continues to the $\epsilon$-annulus $A_j(\epsilon)$ and satisfies (\ref{hol est}) there.  Then the RHS of (\ref{cont2}) can be written in the form
 $$  -\frac{1}{2\pi} \int_{0}^{2\pi} \log |x - e^{is}|  \, \partial_{s} ( \Phi_{\lambda}^{\C}(x) - \Phi_{\lambda}(e^{is})) \, ds + R \phi_{\lambda}(x). $$
 Here, we continue to view $x = e^{iw}  \in \C$ as a  complex variable. Then, writing $ 2 \log |x-e^{is}| = \log (x-e^{is}) + \log (\bar{x} - e^{-is}),$ an integration by parts with respect to $\partial_{s}$  gives
\begin{align} \label{cont3}
  {\mathcal Sl}(g \phi_{\lambda})(x)   &= \frac{1}{2\pi} \Re  \int_{0}^{2\pi}  \frac{ \Phi^{\C}_{\lambda}(x) - \Phi_{\lambda}(e^{is}) }{x-e^{is}}   \, i e^{is}\, ds + R \phi_{\lambda}(x); \,\, x \in A_j^{-}(\epsilon). \end{align}

  By Taylor expansion and Morera's theorem, for each $s \in [0,2\pi],$  the function $f_{\lambda,s}(x):= \frac{ \Phi^{\C}_{\lambda}(x) - \Phi_{\lambda}(e^{is}) }{x-e^{is}}$ has the same analyticity properties as $\phi_{\lambda}^{\C}(x)$ and consequently, $f_{\lambda,s}$ extends to a holomorphic function for all $x \in A_{j}(\epsilon).$ Thus, the first term on the RHS extends harmonically to $x \in A_j(\epsilon).$

  Also, we note that for each $x \in A_{j}(\epsilon)$ the function $F_{\lambda,x}(s)= \frac{ \Phi^{\C}_{\lambda}(x) - \Phi_{\lambda}(e^{is}) }{x-e^{is}}$ extends holomorphically to $s \in [0,2\pi] + i[-\epsilon, \epsilon].$ By Cauchy's formula, we can therefore deform the contour of integration into   $A_j^{-}(\epsilon)$  by letting   $s \mapsto s - i \delta_0$  for any $0< \delta_0 <\epsilon.$ Undoing the integration by parts in (\ref{cont3}) then implies that we can write the harmonic continuation formula for the single layer term in the form:
  \begin{equation} \label{cont4}
   {\mathcal Sl}(g \phi_{\lambda})(x) =  \frac{1}{2\pi} \Re \int_{0}^{2\pi} \log |x - e^{is + \delta_0}|  \,\, g(s-i\delta_0) \, \phi_{\lambda}^{\C}(e^{is +\delta_0 }) ds + R \phi_{\lambda}(x), \,\,\,\, x \in A_j^+(\epsilon), \end{equation}
  where $\log$  in (\ref{cont4}) denotes the principal logarithm (we note that  $|x - e^{is} e^{\delta_0} |  >0 $  when $x \in A_{j}^{+}(\epsilon)$). The formula (\ref{cont4}) is useful for estimating the harmonic continuation of ${\mathcal Sl}(g \phi)$ into the exterior of $D$ (where it can blow-up exponentially in $\lambda).$ 
Indeed, since   $  | \log (x -e^{is + \delta_0})| = O(1)$  uniformly for $ (x,s) \in A_j^{+}(\delta_0/2) \times [0,2\pi]$, it follows from (\ref{cont4}) and (\ref{hol est}) that
\begin{equation} \label{cont5}
 |  {\mathcal Sl}(g \phi_{\lambda})(x) | \lessapprox e^{\lambda \delta_0}  \| \phi_{\lambda} \|_{L^2(\partial D_j)},
 \end{equation}
 uniformly for $ x \in A_{j}^+(\delta_0/2).$
  
 The double layer term ${\mathcal Dl}(\phi_{\lambda})(x)$ is analyzed similarily. In this case, we again rewrite the integral in complex form. Given the parametrization $[0,2\pi] \ni s \mapsto e^{is} = q(s)$ of $\partial D_j,$ the unit normal is $\nu(q(s)) = e^{is}$ and so,
  $$ \partial_{\nu(q(s))} G(x,q(s)) =  \partial_{\nu(q(s))} \log |x - q(s)| =    \Re  \Big( \frac{  e^{is} }{ x -e^{is}}  \Big).$$
  As a result, we can write in complex form:
  \begin{equation} \label{complex double layer}
  {\mathcal Dl}(\phi_{\lambda})(x) =  \Re  \int_{0}^{2\pi}   \phi_{\lambda}(q(s)) \cdot  \Big( \frac{q(s)}{x-q(s)} \Big)   \, ds.\end{equation}
  
   Next,  writing $ \frac{-1}{q'(s)} \partial_s    \log (x-q(s)) = \frac{1}{x-q(s)},$ one integrates by parts  in $s$ with the effect of replacing the $O(|x-q(s)|^{-1})$ singularity in (\ref{complex double layer}) with a  $\log |x-q(s)|$-type singularity. In the process, one creates an extra factor of $\lambda$ coming from differentiation of $\phi_{\lambda}(q(s)).$  To see this, we note that from Proposition \ref{boundary lemma}, $P = \oplus_{j} M_j + R$ with $R \in \Psi^{-\infty}_{a}$ and  $P \phi_{\lambda} = \lambda \phi_{\lambda}.$ Since $M_j (e^{ims_j}) = |m| e^{ms_j},$ for all $m$, it follows from Fourier expansion that 
$$ \|\partial_s \phi_{\lambda}  \|_{L^2(\partial D_j)}  = \| M_j \phi_{\lambda} \|_{L^2(\partial D_j)} = \| (P - R) \phi_{\lambda} \|_{L^2(\partial D_j)} = O(\lambda) \| \phi_{\lambda} \|_{L^2(\partial D_j)}.$$  
   
    So, modulo replacing $\phi_{\lambda}(q(s))$ with $\partial_{s} \phi_{\lambda}(q(s)),$ the analysis proceeds as in the single-layer case and $|{\mathcal Dl}( \phi_{\lambda})(x)|$  also  satisfies the bound 
 \begin{equation} \label{cont6}
 |  {\mathcal Dl}(\phi_{\lambda})(x) | \lessapprox  \lambda e^{\lambda \delta_0}  \| \phi_{\lambda} \|_{L^2(\partial D_j)},
 \end{equation}
 uniformly for $ x = e^{iw} \in A_{j}^+(\delta_0/2).$
 
 Finally, we note that for $k \geq 1,$   $\| {\mathcal Sl}(g \phi_{\lambda}) \|_{C^k}$ and  $\| {\mathcal Dl}(\phi_{\lambda}) \|_{C^k}$ are estimated in the same way, except that one must first integrate by parts $k$ times in $s$ to compensate for the $k$-derivatives in $x.$ This creates an extra multiplicative factor of $\lambda^k.$ As a consequence of (\ref{cont1})-(\ref{cont6}), we have proved Proposition \ref{cont prop}. \end{proof}



\section{Approximation of Steklov Eigenfunctions}

In this section, we construct approximations to our Steklov eigenfunctions, which we will use in our study of their nodal sets. First we use Proposition \ref{boundary lemma} and our functions $e_{m,j}$ to construct these approximations on the boundary. Then we show how to construct quasimodes for Steklov eigenfunctions in the interior.

 \subsection{Approximation of boundary eigenfunctions} 
 
As a first step, Proposition \ref{boundary lemma} allows us to show that the boundary Steklov eigenfunctions $\varphi_{\lambda_n}$ can be approximated up to exponentially decaying error by linear combinations of the quasimodes $e_{m,j}$. More specifically, one has the following 
\begin{lem}\label{quasimodeapprox} There exist functions $\psi_{\lambda_n}$ and $f_{\lambda_n}$ in $C^{\infty}(\partial D)$, $n=0,1,2,\dots$, such that $\varphi_{\lambda_n}=\psi_{\lambda_n}+f_{\lambda_n}$ for each $n$, and furthermore:
\begin{enumerate}
\item Let $\psi_{\lambda_n,j}:= \psi_{\lambda_n} |_{\partial D_j}$, and suppose $n$ is sufficiently large. Then for each boundary component $\partial D_j$, either $\psi_{\lambda_n,j}=0$ or there exist $m_{n,j}\in\mathbb Z$ as well as $b_{n,j,+}$ and $b_{n,j,-}\in\mathbb R$ such that
$$\psi_{\lambda_n,j} = b_{n,j,+}e^{2\pi im_{n,j}s_j/L_j}+b_{n,j,-}e^{-2\pi im_{n,j}s_j/L_j},$$
with $$|\lambda_n-\frac{2\pi m_{n,j}}{L_j}|\leq\frac{2\pi}{L},\ L:=\max\{L_1,\dots,L_k\}.$$
\item For each $\ell \in \Z^+,$ there exist constants $C_\ell>0$ and $\tau_0>0$  such that $$||f_{\lambda_n}||_{C^\ell(\partial D)}\leq C_\ell \lambda_n^\ell e^{-\tau_0 \lambda_n}.$$
\end{enumerate}
\end{lem}
\begin{rem}\label{splitremark} 
Since $\varphi_{\lambda_n}$ are $L^2$-normalized, for sufficiently large $n$, there exist boundary components for which $\psi_{\lambda_n,j}$ are not too small - in particular, large compared to $|f_{\lambda_n}|_{\partial D_j}$. These boundary components will be called \emph{dominant} for $u_{\lambda_n}$ (the precise definition follows in section \ref{sectiondomres}). Those components on which $\psi_{\lambda_n,j}$ are small, or vanish identically, will be called \emph{residual.}\end{rem}

Another consequence of Proposition \ref{boundary lemma} is eigenvalue asymptotics with exponentially decaying error:
\begin{prop}\label{cor:eigenvaluedecay} There exist $C>0$ and $\tau_1>0$ such that $|\lambda_n-\mu_n|\leq Ce^{-\tau_1 n}$.
\end{prop}
The proofs of Lemma \ref{quasimodeapprox} and Proposition \ref{cor:eigenvaluedecay} are relatively standard, but involve some technical linear algebra and are thus deferred to the Appendix.

\begin{rem} 
From Proposition \ref{cor:eigenvaluedecay}, for each eigenvalue $\lambda_n$ with $n$ sufficiently large, there exists at least one boundary component $\partial D_j$ and an integer $m_{n,j}$ for which
$$ |\lambda_n - \frac{2\pi m_{n,j}}{L_j} |  = O(e^{-\tau_1 n}).$$ \end{rem}

\subsection{Interior quasimodes} \label{quasimodes} In this section, we extend our approximate boundary eigenfunctions $\psi_{\lambda_n}$ to the interior and show that the extensions are good approximations to the interior Steklov eigenfunctions.
Recall that
$$ \psi_{\lambda_n}=\sum_{j=1}^k\psi_{\lambda_n,j}=\sum_{j=1}^k  \big( b_{n,j,+}e^{2\pi im_{n,j}s_j/L_j}+b_{n,j,-}e^{-2\pi im_{n,j}s_j/L_j} \big).$$
We will extend each $\psi_{\lambda_n,j}$ to an almost harmonic function $\bar u_{n,j}\in C^{\infty}(\bar D)$, then sum them to create our global interior quasimode $\bar u_n$.

As a first step, we define the rate constant
\begin{equation} \label{tau}
\tau:= \min_{j=1,...,k} \left\{ \tau_0, \frac{2\pi}{L_j} \tilde{\tau} \right\},
\end{equation}
where $\tau_0$ is the exponential rate for the approximation in Lemma \ref{quasimodeapprox} and $\tilde{\tau}$ is  the minimal analytic modulus in Proposition \ref{boundary lemma}.

%

To construct the needed approximately harmonic functions, let $s_j: \partial D_j \to [0,L_j]$ be the arc length function (with metric $g$) along the boundary circle $\partial D_j.$   Without loss of generality, we assume here that the circle $\partial D_j$ is centered at the origin.   Identifying $\partial D_j$  with $\R/2\pi \Z$ via the parametrization $\theta: [0,2\pi]  \to \partial D_j,$ we  let $$\Gamma_j := \min_{\theta \in [0, 2\pi]} s_{j}'(\theta).$$ Since $s_j$ is monotone increasing, each of the constants $\Gamma_j >0.$ We let $s_j^{\C}(z)$ be the holomorphic continuation to the strip $S_{\tau}:= [0,2\pi] + i (-\tau,\tau)$ and   put $N_j:= \max_{\theta + i \xi \in S_{\tau}} |\partial_{\theta}^{3} \Im s^{\C}_j(\theta + i \xi)|.$  We now fix a constant $\delta >0$ once and for all:
\begin{equation} \label{delta defn}
\delta = \delta(\tau):= \min_{j=1,\dots,k}  \{   \Gamma_j \tau/2 - N_j \tau^3,  \tau /2 \}, \end{equation}
with $\tau>0$ defined by (\ref{tau}).
By possibly shrinking $\tau >0$ further, we can (and do) assume from now on that $ \delta >0.$

In the following, we identify holomorphic extensions of $2\pi$-periodic $C^{\omega}$-functions to a strip over $[0,2\pi]$ with an annular neighbourhood of $|z| =1$ via the conformal map $z\mapsto e^{ i z \frac{2\pi}{L_j}}, \,\, z\in [0,2\pi] + i(-\tau,\tau).$ For each $j$, let
$$e_{n,j}^{\C}(z): = e^{ \frac{2\pi}{L_j} i m_{n,j} s_j^{\C}(z)}, \,\,\, z \in [0,2\pi] + i (-\tau, \tau).$$
Then,  natural harmonic extensions to the strip are given by 
\begin{equation} \label{int qm 0}
u_{n,j,+}(z) := \Re e^{\C}_{n,j}(z), \,\, u_{n,j,-}(z) = \Im e^{\C}_{n,j}(z), \,\,\, z \in [0,2\pi] + i (-\tau, \tau). \end{equation}
With $z = \theta + i \xi,$ we can write
$$u_{n,j,+}(z) = e^{- \frac{2\pi}{L_j} m_{n,j} \Im s_j^{\C}(z)} \, \cos \big(  \frac{2\pi}{L_j} m_{n,j} \Re s_j^{\C}(z) \big),$$
$$  u_{n,j,-}(z) = e^{- \frac{2\pi}{L_j} m_{n,j} \Im s_j^{\C}(z)} \,\, \sin \big( \frac{2\pi}{L_j} m_{n,j} \Re s_j^{\C}(z) \big), $$
where, by Taylor expansion, $\Re s_j^{\C}$ and $\Im s_j^{\C}$ are real-analytic functions with
\begin{align} \label{int qm 1}
\Im s_j^{\C}(\theta, \xi) &= s'(\theta) \xi + O(\xi^3), \nonumber \\
\Re s_j^{\C}(\theta, \xi) &= s(\theta) + O(\xi^2). \end{align}
As above, since the arclength function $s(\theta)$ is strictly increasing, 
\begin{equation}
\label{arclengthbound}
s'(\theta) \geq  \Gamma_j >0.
\end{equation}

To define the global quasimode corresponding to $u_{\lambda_n}$ we  let $\chi_j \in C^{\infty}_0(\R^2; [0,1])$ be standard cutoff equal to $1$ in annular neighbourhood of $\partial D_j$ of width $\tau/2$ and vanishing outside an annulus of width $\tau.$ Then, with $x = e^{iz},$ and $j \in \{1,...,k \}$ corresponding to an  {\em outer}  boundary component and with $b_{n,j,\pm}$ as in Lemma \ref{quasimodeapprox}, we define the global interior quasimode to be
\begin{multline} \label{int qm defn}
\bar u_{n,j}(x) := \chi_j(x) \cdot ( b_{n,j,+}  \, u_{n,j,+}(z) + b_{n,j,-} \, u_{n,j,-}(z) )  \\
= \chi_j(x) \cdot \Big(  b_{n,j,+} \, e^{-\frac{2\pi}{L_j} m_{n,j} \Im s_j^{\C}(z) } \cos \big( \frac{2\pi}{L_j} m_{n,j} \Re s_j^{\C}(z) \big)  \\ + b_{n,j,-} \, e^{ - \frac{2\pi}{L_j} m_{n,j} \Im s_j^{\C}(z) }  \sin  
\big( \frac{2\pi}{L_j} m_{n,j} \Re s_j^{\C}(z)  \big) \Big). \end{multline}

 We note that in the case of an outer boundary component, the point $ x = e^{iz} $ with $ \Im z = \xi >0$   will lie in an interior collar neighbourhood of the boundary circle.  In view of (\ref{int qm 1}), for such points, $ \Im s^{\C}_j(z) >0$  and so, the quasimode (\ref{int qm defn})  decays  exponentially in $n$. 


When $\partial D_j$ with $j \in \{1,...,k\}$ is an   {\em inner}  boundary component, we replace $e_{n,j}^{\C}(z)$ above with $\tilde{e}_{n,j}^{\C}(z):=e^{- \frac{2\pi}{L_j} i m_{n,j} s_j^{\C}(z) }$ and replace $u_{n,j,\pm}(z)$ in (\ref{int qm defn}) with $\Re \tilde{e}_{n,j}^{\C}(z)$ and $\Im \tilde{e}_{n,j}^{\C}(z)$ respectively.  In this case, the corresponding quasimodes decay exponentially in an interior collar neighbourhood where $\Im s^{\C}_j(z) < 0.$

The function $\bar u_{n,j} \in C^{\infty}(\bar{D})$ is approximately harmonic in $D$ and agrees with the boundary quasimode $\psi_{\lambda_n,j}$ on $\partial D_j.$  Indeed, since $\Delta u_{n,j,\pm} = 0,$  by Leibniz' rule for the Laplacian 
$$ \Delta \bar u_{n,j} = \Delta [ \chi_j \cdot ( b_{n,j,+} u_{n,j,+} + b_{n,j,-} u_{n,j,-} ) ]$$
$$ = 2 \nabla_x \chi_j \cdot 
\nabla  ( b_{n,j,+} u_{n,j,+} + b_{n,j,-} u_{n,j,-} )  + \Delta \chi_j \,  ( b_{n,j,+} u_{n,j,+} + b_{n,j,-} u_{n,j,-}).$$
 Since the derivatives of $\chi_j$ are supported in an annular neighbourhood of $\partial D_j$ where $ \tau/2 < |x| < \tau,$ it follows in view of (\ref{int qm 1}) that with $\delta = \delta(\tau)$ as in (\ref{delta defn}),
 $$ \| \Delta_x \bar u_{n,j} \|_{C^0(\bar D)} = O(\lambda_n e^{-\delta \lambda_n})$$ and similarily,
 
\begin{equation} \label{int qm 2}
\| \Delta_x \bar u_{n,j} \|_{C^{\ell}(\bar{D})}   = O_{\ell}( \lambda_n^{\ell + 1} \, e^{- \delta \lambda_n}), \quad \ell \in \Z^+.
\end{equation}

Now for each $n$, we  define the {\em interior quasimode}
\begin{equation} \label{int qm 3}
\bar u_n:=\sum_{j=1}^k \bar u_{n,j}. \end{equation}
\noindent As we now prove, the function $\bar u_n$ in (\ref{int qm 3}) is the quasimode approximation to the Steklov eigenfunction $u_{\lambda_n}$ in a collar neighbourhood of the boundary $\partial D$ that we seek. 

\begin{lem}   \label{APPROX} For any fixed $\ell \in \Z^+,$ there exist constants $C_j(\ell)>0; j=1,2,$  such that $$||u_{\lambda_n} -\bar u_n||_{C^\ell(\overline D)}\leq C_1(\ell) \lambda_n^{C_2(\ell)} e^{-\delta \lambda_n}.$$\end{lem}
\begin{proof}

For each $n$, let 
\begin{equation}
\bar f_n:=u_{\lambda_n}-\bar u_n,
\end{equation} so that $\bar f_n$ is the error in the interior quasimode approximation. To $||\bar f_n||_{C^\ell(\bar D)}$ we note that by construction, $\bar u_n=\psi_{\lambda_n}$ on $\partial D$, so $\bar f_n|_{\partial D}=f_{\lambda_n}$, which by Lemma \ref{quasimodeapprox} and the definition of $\delta$ shows that 
\begin{equation}\label{bdy estimates}
||\bar f_n|_{\partial D}||_{C^\ell (\partial D)}\leq C_k \lambda_n^\ell e^{-\tau_0 \lambda_n}\leq C_\ell \lambda_n^\ell e^{-\delta \lambda_n}.
\end{equation}
Moreover, both $\bar u_n$ and $\bar f_n$ are approximately harmonic. Indeed, in view of \eqref{int qm 2} and \eqref{int qm 3}, we see
\begin{equation} \label{qm bounds}
\| \Delta_x \bar f_{n} \|_{C^{\ell}(\bar{D})}  =
\| \Delta_x \bar u_{n} \|_{C^{\ell}(\bar{D})}  \leq C_{\ell}\lambda_n^{\ell+1}e^{-\delta\lambda_n}.
\end{equation}

To begin, we claim that
\begin{equation}\label{czeroestimate}
\| \bar f_n\|_{C^0(\bar D)}\leq C\lambda_n e^{-\delta\lambda_n}.
\end{equation}
Indeed, let $\Phi$ be  the solution of $\Delta\Phi= -1$ on $D$ satisfying the Dirichlet boundary condition $\Phi|_{\partial D}=0$. Then, by the maximum principle, $\Phi >0$ in $D.$  We define
$$\bar g_n:=C_0e^{-\delta\lambda_n}+C_1\lambda_ne^{-\delta\lambda_n}\Phi > 0.$$
By \eqref{bdy estimates} and \eqref{qm bounds}, $ \bar f_n \leq \bar g_n$ on $\partial D$ with $\Delta \bar f_n \geq \Delta \bar g_n.$ Similarily, $\bar f_n \geq - \bar g_n$ on $\partial D$ with $ \Delta (-\bar g_n) \geq \Delta \bar f_n$ on $D$. Consequently, again by the maximum principle,
$$ | \bar f_n | \leq | \bar g_n | = \bar g_n\mbox{ on }\bar D,$$
which proves \eqref{czeroestimate}.  

To treat the cases where $m \geq 1,$ we  note that by standard elliptic estimates (\cite{Kr}, Theorem 5.4.1), for any $m \in \Z^+$ even, 
\begin{equation} \label{elliptic1}
\| \bar f_n \|_{H^{m}(\bar D)} \leq C_m \big( \| \Delta^{\frac{m}{2}} \bar f_n \|_{L^2(\bar D)} + \|  \bar f_n \|_{L^2(\bar D)} + \| \bar f_n|_{\partial D} \|_{H^{m - \frac{1}{2}}(\partial \bar D)} \big). \end{equation}
 Substitution of \eqref{czeroestimate} and (\ref{bdy estimates}) in (\ref{elliptic2}) gives
 \begin{equation} \label{elliptic2}
\| \bar f_n \|_{H^{m}(\bar D)} \leq C_m \big( \| \Delta^{\frac{m}{2}} \bar f_n \|_{L^2(\bar D)} + \lambda_n e^{-\delta \lambda_n} + \lambda_n^{m - \frac{1}{2}} e^{-\delta \lambda_n} \big). \end{equation}
Finallly, using (\ref{qm bounds}) and the fact that $||\Delta^{m/2}\bar f_n||_{L^2(\bar D)}\lessapprox ||\Delta \bar f_n||_{C^{m/2-1}(\bar D)}$, it follows that 
$$ \| \bar f_n \|_{H^{m}(\bar D)} = O_{m}( \, ( \lambda_n^{\frac{m}{2}} +\lambda_n^{m - \frac{1}{2}} + \lambda_n \,) \, e^{-\delta \lambda_n} \,) = O_{m} ( \max \{ \lambda_n^{m-\frac{1}{2}}, \,\lambda_n \} \, e^{-\delta \lambda_n} ).$$
 The proof then follows by an application of the Sobolev lemma which gives
 $$ \| \bar f_n\|_{C^{\ell}(\bar D)} \lessapprox \| \bar f_n\|_{H^m(\bar D)}, \quad m > \ell + 1.$$
\end{proof}



\section{Proof of Theorem \ref{main}}
We now prove Theorem \ref{main}; first we prove lower bounds on nodal length and then proceed to upper bounds. Recall that $u_\lambda$ is the Steklov eigenfunction on $\Omega$ with eigenvalue $\lambda$, and $v_{\lambda}=f^*u_{\lambda}$ is the corresponding eigenfunction on $D$. Throughout, we denote the nodal length of a function $u$ on a set $A$ by $\mathcal L(u,A)$. Note that since the Steklov eigenfunction in the conformal model is  $v_\lambda=f^*u_\lambda$, the ratio of $\mathcal L(u_{\lambda},A)$ and $\mathcal L(v_{\lambda},f^*(A))$ is bounded from above and below by geometric constants independent of $\lambda$ and of $A\subset\Omega$; hence we may work with the conformal model $D$ without creating problems for our estimates. To simplify the writing, we abuse notation somewhat and henceforth denote both the Steklov eigenfunction and its image in the conformal model simply by $u_{\lambda}.$

\subsection{Dominant and residual boundary components}\label{sectiondomres}
Our first step is to group the boundary components $\partial D_j; j=1,...,k,$ into {\em dominant} and {\em residual} categories. The idea here is that since the basic quasimodes approximate actual Steklov eigenfunctions to $O(e^{-\tau \lambda})$ in $C^k$ norm (by Lemma \ref{quasimodeapprox}), eigenfunctions that have $L^2$-norm less than $e^{-\tau \lambda}$  along a boundary component have no meaningful quasimode approximations. We note that since $\lambda_n \approx n$ (see section \ref{boundvalues}), the bounds in Lemma \ref{quasimodeapprox} in terms of $e^{-\tau_0 n}$ are comparable to ones in terms of $e^{-\tau \lambda_n}$ by possibly adjusting the constant $\tau>0.$ We choose to work here with the latter. The simple example of the annulus (see Example \ref{annulus}) shows that   $L^2$-norms of Steklov eigenfunctions can indeed be exponentially small with  $\| \phi_{\lambda} \|_{L^2(\partial D_j)} \sim e^{-\lambda C}$ along certain  boundary circles with some $C>0$. Moreover, it is not clear in general that $C>0$ will be smaller than the exponential  rate $\tau >0$ appearing in the quasimode approximation in Lemma \ref{quasimodeapprox}.

\begin{defn} \label{dominant}  Recall the definition of $\delta$ from \eqref{delta defn}. The boundary component $\partial D_j$ is said to be {\em dominant} for $u_{\lambda}$ provided
$$ \| u_{\lambda} \|_{L^2(\partial D_j)} \geq e^{-\delta \lambda/3}.$$
Otherwise, it is said to be {\em residual}.\end{defn}
We refer to Example \ref{annulus} for an illustration of the concepts of dominant and residual boundary components. If a surface has a single boundary component, it is obviously dominant for each $u_\lambda$.

A key observation about our dominant and residual boundary components is the following. Recall the approximation in Lemma \ref{quasimodeapprox}, where the approximate boundary eigenfunction $\psi_{\lambda_n}$ is a linear combination of trigonometric polynomials on each boundary component $\partial D_j$ with coefficients $b_{n,j,\pm}$. The following is an immediate consequence of our definition of dominant and residual, combined with the error bounds in Lemma \ref{quasimodeapprox}:
\begin{prop}\label{prop:meaningofdominant} If $\partial D_j$ is dominant for $u_{\lambda_n}$, then 
$$|b_{n,j,+}|^2 + |b_{n,j,-}|^2 \geq e^{- \frac{2 \delta \lambda_n} {3 } }  + O(e^{-4\delta \lambda_n})  \geq \frac{1}{2} e^{- \frac{2 \delta \lambda_n} {3 } }$$
as long as $n$ is sufficiently large.
\end{prop}

\begin{rem} We note that even though $L^2$-mass can be exponentially small along boundary components, it has recently been proved by a Carleman argument applied to Steklov eigenfunctions on the boundary $\partial \Omega$ (\cite[Theorem 1]{Zh}), that they satisfy quantitative unique continuation in any dimension. More precisely, there is a constant $C= C(\Omega)>0,$ such that
\begin{equation} \label{uc}
\| \phi_{\lambda} \|_{L^2(\partial \Omega_j)} \geq e^{-C \lambda} \,\,\, \forall j =1,...,k. \end{equation} 
where $\partial \Omega_j; j=1,\dots,k$ denote disjoint subsets of $\partial \Omega.$\end{rem}

Note that as $\lambda \to \infty$, the dominant components of the boundary change depending on $\lambda$. However, there are only finitely many configurations of dominant and residual components, so by passing to subsequences we may consider each configuration separately. Therefore, without loss of generality, we assume that the boundary components 
$\partial D_{j}; j=1,\dots,P$ are dominant and the remaining ones with $j=P+1,\dots, k$ are residual. 
Given the $L^2$-normalization condition (\ref{normalization}), we may also assume that $||u_{\lambda}||_{\partial D_1}\geq\frac{1}{2k}$ (at least one dominant boundary component has this property for sufficiently large $\lambda$, and there are only finitely many configurations). In the following, we denote by $A_j(\alpha)$ an $\alpha$-width annular neighbourhood of the dominant boundary component $\partial D_j$.


\subsection{Nodal length near a dominant boundary component}
\label{sss}
We now obtain optimal upper and lower bounds for the nodal length of a Steklov eigenfunction $u_{\lambda}$ near a dominant boundary component.  
\begin{lem} Let $\partial D_j$ be a dominant boundary component  and  let $ \tau>0$ and $\delta(\tau) >0 $ be defined as above (see \ref{delta defn} and \ref{tau}). Then,   there exists $\alpha = \alpha(\tau) >0$ and geometric constants $C_j = C_j(\partial D, D); j=1,2,$   such that, for any Steklov eigenfunction $u_{\lambda}$  \[C_1 \lambda\leq\mathcal L(Z_{u_{\lambda}}\cap A_j(\alpha))\leq C_2\lambda.\]
\end{lem}
\begin{proof} We use the decomposition $u_{\lambda_n}=:\bar u_{n,j}+\bar f_{n,j}$, with $C^k$-bounds on $\bar f_{n,j}$ as in Lemma \ref{APPROX} (note that $\bar f_{n,j}=\bar f_n$ in $A_j(\alpha)$). Rescale all of these functions by multiplying by $C>0$, where $C$ is chosen so that $b_{n,j,+}^2+b_{n,j,-}^2=1$. By Proposition \ref{prop:meaningofdominant}, $C\leq 2e^{\frac{2\delta\lambda_n}{3}}$ for sufficiently large $n$. None of the zero sets change, and by our observations about the size of $C$, we now have \begin{equation}\label{rescaledfnbounds}||f_{n,j}||_{C^k}= \mathcal O_k (e^{-\delta \lambda_n/2 + \delta \lambda_n/3}) = \mathcal O_k(e^{-\delta \lambda_n/6}).\end{equation}

For simplicity, assume that $\partial D_j$ is the outer boundary (a similar argument works for the inner boundaries).  As in section \ref{quasimodes}, we use the coordinates $(\theta,\xi) \in [0,2\pi] + i (0, \delta)$ in the  $\delta$-strip model of the annulus $A(\alpha)$ with complex variable $z = \theta + i\xi.$ Here, $\xi >0$ corresponds to the interior of $D$ where eigenfunctions decay exponentially in $\lambda_n.$   A direct computation with the quasimodes of \eqref{int qm defn} gives
\begin{equation}\label{trigfunctionmagic0}\bar u_{n,j}^2+ \lambda_n^{-2}  |\partial_{\theta}\bar u_{n,j}|^2  \gtrapprox e^{- \frac{4\pi m_{n,j}}{L_j}  \Im s^{\C}_j(\theta,\xi) }.  \end{equation}

We recall from (\ref{int qm 1}) that   $ \Im s^{\C}_j(\theta,\xi) = s_j'(\theta) \xi + O(\xi^3)$ with $ \tilde{\Gamma_j} \geq s_{j}'(\theta) \geq \Gamma_j>0$ and so $\Im s_j^{\C}(\theta,\xi) \approx \xi $ for $0 \leq \xi < C_1(\tau)$ with $C_1(\tau)$ small.  Thus, choosing $C_1(\tau) >0$ sufficiently small, one can arrange that 
$$ 0 \leq \Im s_j^{\C}(\theta,\xi) < 2 C_1(\tau) \tilde{\Gamma_j}.$$

Then, by possibly shrinking $C_1(\tau)>0$ in (\ref{trigfunctionmagic0}) further, one can arrange that for $0 \leq \xi \leq C_1(\tau),$

\begin{equation}\label{trigfunctionmagic}\bar u_{n,j}^2+ \lambda_n^{-2}  |\partial_{\theta}\bar u_{n,j}|^2  \gtrapprox e^{- \delta  \,  m_{n,j}/ 5} \gtrapprox e^{-\delta \lambda_n/5}.  \end{equation}

In the last inequality in (\ref{trigfunctionmagic}), we have use the fact that $|m_{n,j}-\lambda_n|<2\pi/L$.

 Now, fix   $\alpha =  C_1(\tau)$  and  let $E_{\lambda_n}$ be the subset of $A_j(\alpha) = \{ (\theta,\xi); 0 \leq \xi < \alpha\}$ where $|\bar u_{n,j}|^2 \leq e^{-\delta \lambda_n/4}$.  By formula  \eqref{rescaledfnbounds} , for $\lambda_n$ sufficiently large, $Z_{u_{\lambda_n}}\cap A_j(\alpha)\subseteq E_{\lambda_n}$.
But using \eqref{trigfunctionmagic}, for $\lambda_n$ sufficiently large,
\begin{align} \label{ift}
|\partial_{\theta}\bar u_{n,j}(\theta,\xi)| &\geq C \lambda_n e^{- \frac{2\pi m_{n,j}}{L_j} \Im s^{\C}_j(\theta,\xi)} \geq C' \lambda_n e^{-\delta \lambda_n/10}, \,\, (\theta,\xi) \in E_{\lambda_n}.\end{align}
Using \eqref{rescaledfnbounds}, eigenfunctions are approximated by quasimodes in $C^k$-norm to $O(e^{-\delta \lambda_n/6})$-error and so, it follows from (\ref{ift}) that for the actual eigenfunctions,
\[|\partial_{\theta} u_{\lambda_n}(\theta,\xi)|\geq  C \lambda_n e^{- \frac{2\pi m_{n,j}}{L_j} \Im s^{\C}_j(\theta,\xi)} \geq C'\lambda_n  e^{- \delta \lambda_n/10}, \,\, (\theta,\xi) \in E_{\lambda_n}. \]

So on each connected component of $E_{\lambda_n}$, $u_{\lambda_n}$ is either monotonically increasing or monotonically decreasing in $\theta$. Since for each fixed $\xi$, $\bar u_{n,j}$ has exactly one zero in each connected component of $E_{\lambda_n}$, so must $u_{\lambda_n}$ for each fixed $\xi$. From (\ref{ift}), it follows by the analytic implicit function theorem that the nodal set $Z_{u_{\lambda_n}}\cap A_j(\alpha)$ is a union of $C^{\omega}$ curves, one in each connected component of $E_{\lambda_n}$, which are graphs with dependent variable $\theta$ and independent variable $\xi$. There are precisely $2\mu_n$ zeroes of $\bar u_{n,j}$, so we can write $Z_{u_{\lambda_n}}\cap A_j(\alpha)$ as a union of graphs of $C^{\omega}$  functions $f_1(\xi),\dots,f_{2\mu_n}(\xi)$.

It remains to control the geometry of these graphs. To do this, we claim that there exists a constant $C'>0$ such that
\begin{equation}\label{gradratio}  |\partial_{\xi} u_{\lambda_n}|\leq C'|\partial_\theta u_{\lambda_n}|\textrm{ in }E_{\lambda_n}. \end{equation}

Given (\ref{gradratio}), it follows by the chain rule that $|f_i'(\xi)|$ is uniformly bounded in $\lambda_n$ for each $\xi$. Therefore the arc length of each $f_i(\xi)$ is uniformly bounded above and below. Since there are $2\mu_n$ such graphs, and $2\mu_n\sim 2\lambda_n$, the result follows.

To prove (\ref{gradratio}), we note that by (\ref{ift}),
 \[|\partial_{\theta}\bar u_{n,j}(\theta,\xi) |\geq C_2 \lambda_n e^{-  \frac{2\pi m_{n,j}}{L_j} \Im s_j^{\C}(\theta,\xi)}, \,\, (\theta,\xi) \in E_{\lambda}.\]
 We also note that from \eqref{int qm 1}, $\partial_{\xi} \Re s_j^{\C}(\theta,\xi) = O(\xi)$ where $|\xi| < \alpha(\tau) \ll1$ in $E_{\lambda_n}$ and $\partial_{\xi} \Im s_j^{\C}(\theta,\xi) = s_j'(\theta) + O(\xi^2)$ with  $s_j'(\theta) \geq \Gamma_j>0.$ Thus, by another direct computation with the quasimodes \eqref{int qm defn}, it follows  that
\[|\partial_{\xi} \bar u_{n,j}(\theta,\xi)|\leq C_2 \lambda_n e^{- \frac{2\pi m_{n,j}}{L_j} \Im s_j^{\C}(\theta,\xi)}, \,\, (\theta,\xi) \in E_{\lambda_n}.\]
The analogue of \eqref{gradratio} is therefore immediately true for the quasimodes $\bar u_{n,j}$ in place of the eigenfunctions $u_{\lambda_n}$. Transferring to $u_{\lambda_n}$ via Lemma \ref{APPROX} introduces errors that are $ {\mathcal O}(e^{-\delta\lambda_n/6}).$  However, in view of (\ref{ift}),  these are lower-order than $e^{-\frac{2\pi m_{n,j}}{L_j} \Im s_j^{\C}} \gtrapprox e^{-\lambda_n \delta/10}$ in $E_{\lambda_n}.$ Such errors are therefore negligible, so the pointwise estimate in (\ref{gradratio}) is satisfied for the actual eigenfunctions $u_{\lambda_n}.$
\end{proof}

This gives both upper and lower bounds for nodal length near dominant boundary components, and the overall lower bound is an immediate corollary. It remains to estimate, from above, nodal lengths in the interior and near residual boundary components.

\subsection{Estimates for the harmonic extensions of interior eigenfunctions near residual boundary components} To estimate nodal lengths all the way up to residual boundary components, we need to extend our domain slightly near the residual boundary components (independent of $\lambda$) and use  this slightly enlarged domain. In the following, we continue to work in the conformal model. So let $\widetilde D$ be an extension of $D$ to include a ($\lambda$-independent) neighborhood of the residual boundaries. We claim the following:

\begin{lem}\label{extension} Given $\delta >0$ as above, there exists an   open domain $\widetilde D(\delta) \supset \bar{D}$ with smooth boundary containing open neighbourhoods of the residual boundary circles  such that the Steklov eigenfunctions $\varphi_{\lambda}$ all have a harmonic extension to $\widetilde D(\delta)$ denoted by $u_{\lambda}.$ Moreover, in the extended neighbourhoods near the residual components one has the following estimate: 
$$||u_{\lambda}||_{C^1(\widetilde D(\delta)\setminus \bar D)} \leq C e^{-\frac{\delta}{6} \lambda}$$
 for some geometric constant $C>0$ depending on $\delta$ but independent of $\lambda$.
\end{lem}
\begin{proof} This follows from Proposition \ref{cont prop}. Given any residual boundary component $\partial D_j$ and $\delta_0 < \delta,$ consider the annular neighbourhood $A_j(\delta_0) \supset \partial D_j$ with
$$ \| u_{\lambda} \|_{C^1(A_j(\delta_0))} \leq C \lambda^2 e^{ \lambda \delta_0} \| \phi_{\lambda} \|_{L^2(\partial D_j)} = O( \lambda^2 e ^{\lambda ( - \frac{\delta}{3} + \delta_0) )} ).$$
Choose $\delta_0 < \frac{\delta}{6}$ and let $\widetilde D = D \cup ( \cup_{j=P+1}^{k} A_{j}(\delta_0) ).$
\end{proof}
To simplify notation, in the following we denote the extended domain $\widetilde D(\delta)$ simply by $\widetilde D.$

\subsection{Nodal length bounds away from dominant boundary components}
We begin the analysis away from the dominant boundary components by considering simply connected sets in the interior of the extended domain $\tilde D$. In this section, we prove nodal bounds on these sets, then complete the proof with a covering argument. 
\begin{lem}\label{lem:scbound}Let $U\Subset\tilde U$ be simply connected open sets in $\tilde D$, each with $C^{\infty}$ boundary, where:
\begin{enumerate}
\item $U$ is not contained in $\cup_{j=1}^PA_j(\alpha)$, where $\alpha=\alpha(\tau)$ was fixed in section \ref{sss}.
\item For each $j=1,\dots,P$, $\partial\tilde U\cap\partial D_j$ is a nontrivial circular subarc of $\partial D_j$. 
\item \emph{(Technical assumption)} Translating the angle coordinate $\theta$ if necessary, we assume that $\partial\tilde U\cap A_j(\alpha)$, for each $j$, may be parametrized in the form $\{(\theta,\xi);|\theta|\leq\theta_0,\xi=F(\theta)\}$ with $F(\theta)=0$ for $|\theta|<\epsilon_0$, 
$F'(\theta) >0$ for $\theta>\epsilon_0$ and $F'(\theta)<0$ for $\theta< -\epsilon_0$.
\end{enumerate} Then there exists a constant $C$, depending on $U$ and $\tilde U$ but independent of $\lambda$, such that \begin{equation}\mathcal L(u_{\lambda},U)\leq C\lambda.\end{equation}
\end{lem} 
\begin{rem} The technical assumption, made for convenience (see \eqref{universal upper}), says that the boundary of $\tilde U$ approaches each dominant boundary component $\partial D_j$ in a monotone fashion, travels along it for some distance, then departs in a monotone fashion. Note that $F(\cdot)$ measures distance from the boundary $\partial D_j$. See Figure \ref{fig:mainfig} for an illustration of sets $U$ and $\tilde U$ satisfying the hypotheses of the Lemma.
\end{rem}

\begin{proof} 
\emph{Step 1:} Pick $x_0\in U\setminus\cup_{j=1}^{P}A_j(\alpha)$, which is possible by condition (1). Conformally mapping $\tilde U$ to a disk and $x_0$ to the origin keeps the functions harmonic, and changes the length of the nodal set by at most a geometric constant. Therefore, we may assume that $\tilde U$ is a disk of radius 1 and that $x_0$ is the origin.

Note that $U$ is contained in some disk $B_r$ with $r<1$. We cover $B_r$ by finitely many disks $A_1,\dots,A_k$ with centers $x_1,\dots,x_k$ and radii $r_1,\dots,r_k$, where each $A_i$ has the property that $4A_i$, the disk with center $x_i$ and radius $4r_i$, is contained in $B_1$. By finiteness, we only need to prove the bound for each $A_i$. By the nodal measure bound of  Han and Lin \cite[Theorem 2.3.1]{HL},
\begin{equation} \label{HanLin}
{\mathcal L}(u_{\lambda},A_i) \leq C_1 N_{u_{\lambda}}(2A_i),\textrm{ where } N_{u_{\lambda}}(2A_i)=\frac{2r_i \int_{2A_i}|\nabla u_{\lambda}|^2}{\int_{\partial(2A_i)}u_{\lambda}^2}. \end{equation}
Note that Han-Lin's result is stated for a ball of radius 1, but scaling shows that it also holds for a ball of radius $r$, modulo a geometric constant which we absorb into $C_1$. Since $4A_i \subset B_1=\tilde U$ for all $i=1,\dots, k$,  a uniform control estimate on the frequency function (\cite[Theorem 2.2.8]{HL}, see also \cite[Section 3.2.2]{HV}) yields:
$$
N_{u_\lambda}(2A_i) \le C_2 N_{u_\lambda}(\tilde U)
$$
Then, summing up over the disks $A_i$ and using Green's identity as well as  the  Cauchy-Schwarz inequality we get:
\begin{equation}\label{upperboundeq1}
\mathcal L(u_{\lambda},U)\leq   \frac{\int_{\tilde U}|\nabla u_{\lambda}|^2}{\int_{\partial \tilde U}u_{\lambda}^2} \le 
C\frac{||\partial_{\nu}u_{\lambda}||_{\partial\tilde U}||u_{\lambda}||_{\partial\tilde U}}
{||u_{\lambda}||^2_{\partial\tilde U}}.\end{equation}
Here and further on  we use the simplified notation  $||\cdot||_{\tilde U} := ||\cdot||_{L^2(\tilde U)}$.

\smallskip

\emph{Step 2:} Since the argument for each dominant circle is the same, consider here the intersection of the open set $\tilde U$  with the $\tau$-annular neighbourhood of one fixed dominant boundary circle, say $\partial D_1.$ Let us estimate the numerator from above and the denominator from below, beginning with the denominator. We continue to divide through by a constant and without loss of generality assume that $b_{n,1,+}^2 + b_{n,1,-}^2 = 1$,
where $b_{n,1,+}$ and $b_{n,1-}$ are constants defined in Lemma \ref{quasimodeapprox}. 
 Consider the norm over the portion of $\tilde U$ consisting of an arc $\gamma$ along $\partial D_1$.  By the interior quasimode approximation in Lemma \ref{APPROX}, property (2) in  Lemma \ref{lem:scbound} and  a direct computation with the explicit quasimodes in (\ref{int qm defn}), there is a constant $c'>0$ such that
\begin{equation} \label{denom}
||u_{\lambda}||^2_{\partial \tilde{U}}\geq ||\bar u_{n,1} ||^2_{\gamma} - e^{-c' \lambda},
\end{equation}
where $\bar u_{n,1}(x)$ is the global interior quasimode defined in \eqref{int qm defn}. 
Let  the curve segment $\gamma = \{(\theta, \xi); |\theta| < \epsilon_0,   \xi = 0 \}$ with some $\epsilon_0 >0$. Since $\Im s_j^{\C}(\theta, \xi=0)= 0,$
it follows  from Lemma \ref{APPROX} that
$$ ||\bar u_{n,1} ||^2_{\gamma} = \int_{|\theta| < \epsilon_0} \left| \frac{2\pi}{L_1} b_{n,1,+}\,\cos (m_{n,1}\,s(\theta)) + \frac{2\pi}{L_1} b_{n,1,-} \,\sin (m_{n,1}\,s(\theta)) \right|^2 \, d\theta \geq C(\epsilon_0) >0.$$ Together with \eqref{denom} this yields a lower bound 
\begin{equation}
\label{lbd}
||u_{\lambda}||^2_{\partial \tilde{U}}\geq C'(\epsilon_0)>0.
\end{equation}
 To get upper bounds for the numerator in (\ref{upperboundeq1}), we split the integral over $\partial \tilde{U}$ into three pieces: $\partial \tilde{U} = \gamma_1 \cup \gamma_2 \cup \gamma_3.$ The first piece, $\gamma_1 = \partial \tilde{U} \cap A(\alpha)$ is the putative leading term, coming from the part of $\partial \tilde{U}$ inside the annulus $A(\alpha)$ around the dominant circle $\partial D_1.$ As a consequence of our technical assumption (3),

\begin{align} \label{universal upper}
||u_{\lambda}||_{\gamma_1}^2  &=  \int_{|\theta| < \epsilon_0} \left| \frac{2\pi}{L_1} b_{n,1,+}\,\cos (m_{n,1}\,s(\theta)) + \frac{2\pi}{L_1} b_{n,1,-} \,\sin (m_{n,1}\,s(\theta)) \right|^2 \, d\theta  \nonumber \\
& +  \int_{\theta_0 > |\theta| > \epsilon_0} {\mathcal O}(e^{- \Gamma_1 \lambda [ F(\theta) + O(F(\theta)^3) ] })  \, d\theta + O(e^{-\lambda \delta/2})  \approx 1, \end{align}
since $ \int_{\theta_0 > |\theta| > \epsilon_0} {\mathcal O}(e^{- \Gamma_1 \lambda [ F(\theta) + O(F(\theta)^3) ] })  \, d\theta = O(\lambda^{-1})$ for $\delta_0 >0$ small by a change of variables $\theta \mapsto F(\theta)$, using that $|F'(\theta)| >0$ in the range $\epsilon_0 < |\theta| < \delta_0.$

The second piece $\gamma_2 = \partial \tilde{U} \cap [ \bar{D} \setminus A(\alpha)].$ This is the piece of $\partial \tilde{U}$ in the closed domain $\bar{D}$ outside the annulus. Here, we know by   formula  \eqref{rescaledfnbounds}  that 
 with $\alpha = C_1(\tau)>0$ sufficiently small, 
 $$\|u_{\lambda}\|_{\gamma_2}=  \| \bar u_{n,1} \|_{\gamma_2} + O(e^{- \frac{\delta \lambda}{6} }) = O(e^{- \frac{ \Gamma_j \pi C_1(\tau) \, \lambda}{L_j} } ) +  O(e^{- \frac{\delta \lambda}{6} }).$$
 Here, the  first term is bounded using \eqref{int qm defn},  \eqref{int qm 1} and \eqref{arclengthbound}.

Finally, the third piece $\gamma_3 = \partial \tilde U \setminus \bar{D} $ is the part of the boundary of $\tilde{U}$ that is the exterior to $\bar{D}$ in the extended domain $\tilde{D}.$ Using Lemma \ref{extension} we have
$$ \| u_{\lambda} \|_{\gamma_3} = O(e^{-\delta \lambda/6}).$$
Consequently,
\begin{equation}\label{num1}
\| u_{\lambda} \|_{\partial \tilde{U}} = O(1).\end{equation}

For the normal derivative term, the same decomposition argument shows that
\begin{equation} \label{num2}
||\partial_{\nu}u_{\lambda}||_{\partial\tilde U} = O(\lambda). \end{equation}
Then, from (\ref{lbd}), (\ref{num1}) and (\ref{num2})  it follows that
\begin{equation}\label{upperboundeq1.1}\mathcal L(v_{\lambda},U)\leq C\frac{||\partial_{\nu}u_{\lambda}||_{\partial\tilde U}||u_{\lambda}||_{\partial\tilde U}}{||u_{\lambda}||^2_{\partial \tilde{U}}} \leq C \lambda. \end{equation}
This completes the proof of the lemma. \end{proof}

Now we use a covering lemma to finish the proof.

\begin{lem} \label{covering lemma} Suppose $\Omega$ is a connected, compact, smooth Riemannian manifold with smooth boundary $\partial\Omega$ consisting of $k$ connected components $M_1,\dots,M_k$. Then there exists a covering of $\Omega$ by finitely many open sets, each of which is connected, contractible, and satisfies the assumptions on $\tilde U$ in Lemma \ref{lem:scbound}.
\end{lem}
\begin{proof}
First take a finite collection of open balls (and half-balls, near the boundary) $\{U_i\}$ which covers $\Omega$ (compactness guarantees finiteness is possible). Take small enough balls to ensure that $\overline U_i$ is contractible for each $i$ (for example, assume that each ball has radius less than the injectivity radius of $\Omega$). For each $i$, we will find an open set $V_i$ which contains $U_i$, has contractible closure, and moreover intersects each boundary component. The collection $\{V_i\}$ is then the desired cover.

So: begin with a set $U_i$. Since $\Omega$ is connected, for each $j$ between 1 and $k$, let $\gamma_j:[0,1]\to\Omega$ be a smooth path with $\gamma_j(0)\in U_i$ and $\gamma_j(1)\in M_j$.

First, let $t_1=\sup\{t\in [0,1] | \gamma_1(t)\in U_i\}$. Then let $\tilde\gamma_1$ be the portion of the image of $\gamma_1$ corresponding to $t\in[t_1,1]$; let $W_1$ be open such that $\bar W_1$ is a contractible closed neighborhood of $\tilde\gamma_1$, and let $U_{i,1}=U_i\cup W_1$. Now $\bar U_{i,1}$ is contractible (as $\bar W_1$ contracts to $\tilde\gamma_1$, and $U_i\cup\tilde\gamma_1$ contracts to $U_i$), and intersects $M_1$.

Then proceed inductively to define $U_{i,j}$; for each $j$, let $t_j=\sup\{t\in [0,1] | \gamma_j(t)\in U_{i,j-1}\}$, let $\tilde\gamma_j$ be the portion of the image of $\gamma_j$ corresponding to $t\in[t_j,1]$, and let $W_j$ be the interior of a contractible closed neighborhood of $\tilde\gamma_j$. Then let $U_{i,j}=U_{i,j-1}\cup W_j$. At each step, $\bar U_{i,j}$ is contractible, as it contracts to $\bar U_{i,j-1}$ which is assumed contractible by the inductive hypothesis. And $\bar U_{i,j}$ intersects $M_j$. Appropriately smoothing out and adjusting the boundary of $U_{i,k}$ in each neighborhood $A_j(\alpha)$ gives us a new set $V_i$ satisfying all assumptions of Lemma \ref{lem:scbound}, completing the proof. Note in particular that the technical assumption (3) in Lemma \ref{lem:scbound} is essentially local and can always be arranged. \end{proof}

Applying this lemma to $\tilde D$, we obtain a finite collection of open sets $\tilde U_i$, $i=1,\dots,\ell$, which cover $\tilde D$. For each $x\in D\setminus \cup_{j=1}^P A_j(\alpha)$, which is an open subset compactly contained in the interior of $\tilde D$, there is an open ball $U_{x}$ about $x$ compactly contained in some $\tilde U_i$. By compactness, a finite collection of these balls, $\{U_k\}_{k=1}^{l}$, cover $D\setminus\cup_{j=1}^PA_j(\alpha)$. Applying Lemma \ref{lem:scbound} to each such $U_k$ and the $\tilde U_i$ it is contained in gives us, for each $k$,
\[\mathcal L(Z_{u_{\lambda}}\cap U_k)\leq C_k\lambda.\] Finiteness of the covering completes the proof of Theorem \ref{main}.

\subsection{Real analytic Riemannian surfaces: general case} \label{general}
We consider here the general case when $(\Omega,g)$ is a compact real-analytic Riemannian surface with  boundary $\partial \Omega$,  and show that the argument above applies to this setting. It is well known that given a real-analytic Riemannian manifold it is always possible to extend it across its  boundary to a slightly larger open real-analytic Riemannian manifold (see, for instance, \cite[Example 5.50]{EC}).
Let  $(\tilde{\Omega}, \tilde g)$ be such an extension of  the  surface $(\Omega,g)$. On the open Riemann surface $\tilde{\Omega},$ by the Behnke-Stein Theorem (\cite{V} Theorem 9.4.6), there exists a strictly-plurisubharmonic exhaustion function for $\tilde{\Omega}$, and so there is a Green's function $G(z,w) \in {\mathcal D}'(\tilde{\Omega} \times \tilde{\Omega})$ with the property that
$$ \Delta_{g,z} G(z,w) = \delta(w-z), \,\,\, (z,w) \in \tilde{\Omega} \times \tilde{\Omega},$$
and
\begin{equation} \label{general green}
 G(z,w) = -\frac{1}{2\pi} \log |z-w| -a(z,w),  \end{equation}
where $a$ is harmonic in both variables and hence $a \in C^{\omega}(\tilde{\Omega} \times \tilde{\Omega}).$

Next, we note that although we have used the Koebe conformal model in the genus zero planar case for convenience, it is not necessary. Indeed, returning to the initial Steklov eigenfunctions $\phi_{\lambda} \in C^{\omega}(\partial \Omega)$ one can directly derive the potential layer equations used to obtain the crucial quasimode approximations for the Steklov eigenfunctions along $\partial \Omega.$ Therefore, the arguments of section \ref{boundvalues} may be repeated with the only difference that one needs to take into account the contribution of $a(z,w)$. However, since this function is harmonic, its contribution can be absorbed in the last two terms of the expression in Claim \ref{decompint}. Consequently, just as in the planar case,
$$ \frac{1}{2} \phi_{\lambda_n} = \lambda_n G \phi_{\lambda_n} - \lambda_n K_1 \phi_{\lambda_n} - K_2 \phi_{\lambda_n},$$
where the Schwartz kernels of $K_1$ and $K_2$ are elements of $C^{\omega}(\partial \Omega \times \partial \Omega)$ and $G$ is the same as in Claim \ref{decompint}. 
The quasimode approximation results in Proposition \ref{boundary lemma} and Lemma \ref{quasimodeapprox} and  the estimates in Lemmas \ref{APPROX} and \ref{extension}  then follow as in the multiply-connected planar case after decomposing the boundary $\partial \Omega$ into dominant and residual components. This completes the proof in the general case. \qed



\appendix
\section{Proofs of Lemma \ref{quasimodeapprox} and Proposition \ref{cor:eigenvaluedecay}}
Throughout the appendix, for simplicity, we work with a real basis of eigenfunctions.
\subsection{An auxiliary lemma}
Lemma \ref{quasimodeapprox} depends on the presence of recurring spectral gaps, so in order to prove it, we must first prove the following lemma.
\begin{lem}\label{clustering} Fix any sufficiently small $\epsilon>0$. Then there exist pairwise disjoint closed intervals $\mathcal I_i=[A_i,B_i]\subset\mathbb R^{+}$, $i=1,\dots,\infty$, with the following properties:
\begin{enumerate}
\item $\sigma(P)\cup\sigma(M_j)\subseteq\bigcup_{i=1}^{\infty}\mathcal I_i$;
\item $A_{i+1}\geq B_i+\epsilon$ for all $i$;
\item For $i\geq 2$ and all $j$, each $\mathcal I_i$ contains at most one distinct element of $\sigma(M_j)$;
\item For each $i$, there exists $n$ such that $\lambda_n\in\mathcal I_i$ and $\mu_n\in\mathcal I_i$.
\end{enumerate}
\end{lem}
\begin{proof}[Proof of Lemma \ref{clustering}] 

Pick $\epsilon<\frac{\pi}{2kL}$, where $L$ is the maximum of the boundary lengths $L_1,\dots,L_k$. Since $|\lambda_n-\mu_n|\to 0$, there exists $N$ such that if $n\geq N$, then $|\lambda_n-\mu_n|<\epsilon/2$. 

Observe that since $\{\mu_n\}$ is a union of $k$ arithmetic progressions, each with period $\geq 2\pi/L$, there are at most $k$ elements of $\{\mu_n\}$ in any interval of length $\pi/L$. Thus any such interval must contain a gap of length at least $\pi/kL>2\epsilon$ with no elements of $\{\mu_n\}$.
We may therefore choose some $m_2$ so that $m_2>N$ and $\mu_{m_2}\geq\mu_{m_2-1}+2\epsilon$. Consider the interval $[\mu_{m_2},\mu_{m_2}+\pi/L]$ and observe that it must itself contain a gap of length at least $2\epsilon$. Let $n_2\geq m_2$ be such that $\mu_{n_2}$ is the left endpoint of the first such gap; we have $\mu_{n_2}-\mu_{m_2}\leq \pi/L$. Then let $m_3=n_2+1$, so that $\mu_{m_3}$ is the right endpoint of that gap. Choosing $n_3$ as with $n_2$ and iterating this process, we produce $m_2,n_2,m_3,n_3,\dots$.

Now let
\[\mathcal I_1=[0,\mu_{m_2}-3\epsilon/2];\ \mathcal I_i= [\mu_{m_i}-\epsilon/2,\mu_{n_i}+\epsilon/2]\ \forall\, j\geq 2.\]
We claim that these intervals satisfy each property we want. Indeed, by construction, property (2) is automatic. Property (1) follows immediately from the fact that $|\lambda_n-\mu_n|<\epsilon/2$ whenever $n\geq m_2-1\geq N$. To see property (3), note that each $\mathcal I_i$ with $i\geq 2$ has length at most $\pi/L+\epsilon<2\pi/L$, and hence contains at most one element of each arithmetic progression with period $\geq 2\pi/L$. Finally, property (4) is also immediate by construction (note that $\mathcal I_1$ contains $\mu_0=0$). This completes the proof. \end{proof}

We now use these intervals to split the sequence $\{\lambda_n\}$ into pieces with gaps of size at least $\epsilon$ between each. Specifically, fix some $\epsilon$ and let $\mathcal I_i$, $i=1,2,\dots$ be the intervals constructed int Lemma \ref{clustering}. For each $i$, we let $\mathcal L_i$ be the set of all $j$ for which $\lambda_j\in \mathcal I_i$, and we say that 
\begin{equation}\label{equivrel}
j_1\sim j_2\iff\exists\, i\in\mathbb N\mbox{ s.t. }j_1\in\mathcal I_i\mbox{ and }j_2\in\mathcal I_i. 
\end{equation}
Since each interval $\mathcal I_i$ for $i\geq 2$ contains at most $k$ eigenvalues $\lambda_n$ (and each interval, including $\mathcal I_i$, contains at least one), there exists a universal constant $C$ such that
\begin{equation}\label{comparedrates}
j\in\mathcal L_i\Rightarrow j\leq Ci\mbox{ and }i\leq j.
\end{equation}

\subsection{Completing the proofs}

\begin{proof}[Proof of Lemma \ref{quasimodeapprox}] Let $\{\bar e_n\}$ be an orthonormal basis of eigenfunctions for $M$, with eigenvalues $\mu_n$, and note that each $\bar e_{n,j}=\bar e_n|_{\partial D_j}$ is either a trigonometric polynomial with frequency $\mu_n$ or identically zero. Write $a_{j,k}=\langle \varphi_{\lambda_j}, \bar e_k\rangle$. By orthonormality and completeness of the eigenbases, we have
\begin{equation}\label{orthogonal} \bar e_k=\sum_{j=1}^{\infty}a_{j,k}\varphi_{\lambda_j};\quad \varphi_{\lambda_k}=\sum_{j=1}^{\infty}a_{k,j}\bar e_j.\end{equation} 
We claim that if we set
\begin{equation}\label{quasimodebuilding}\psi_{\lambda_n}=\sum_{j\sim n}a_{n,j}\bar e_j,\quad f_{\lambda_n}=\sum_{j\nsim n}a_{n,j}\bar e_j,\end{equation}
where $\sim$ is defined by \eqref{equivrel}, then the conditions of the Lemma are satisfied. Indeed, condition (1) is obvious, since by the definition of our intervals $\mathcal I_i$, the frequency of each $\bar e_j$ with $j\sim n$ is within $2\pi/L$ of $\lambda_n$. It remains to prove (2).

In what follows, we let $C$ and $\tau>0$ be universal constants (which may depend on the geometry of $D$ and on the conformal map from $D$ to $\Omega$) and re-label at will. From Proposition \ref{boundary lemma} and the Weyl asymptotics, we know that there exists $\tau>0$ such that $|P\bar e_j-\mu_j\bar e_j|\leq Ce^{-\tau j}$. Plugging in the first equation in \eqref{orthogonal}, we see that
\[||\sum_{n=1}^{\infty}a_{n,j}(\lambda_n-\mu_j)\varphi_{\lambda_n}||_{\infty}\leq Ce^{-\tau j}\]
and hence that the same is the true of the $L^2$-norm. Summing only over the $n$ with $n\nsim j$, we obtain
\begin{equation}\label{L2bound}(\sum_{n\nsim j}a_{n,j}^2(\lambda_n-\mu_j)^2)^{1/2}\leq Ce^{-\tau j}.\end{equation}
By property (2) of Lemma \ref{clustering}, we know $|\lambda_n-\mu_j|\geq\epsilon$ whenever $n\nsim j$, so
\begin{equation}\label{betterbound}\sum_{n\nsim j}a_{n,j}^2\leq Ce^{-\tau j}.\end{equation}
However, by \eqref{orthogonal}, for each $j$ we have $\sum_{n}a_{n,j}^2=1$. So
\begin{equation}\label{unitvectors}1-Ce^{-\tau j}\leq \sum_{n\sim j}a_{n,j}^2\leq 1.\end{equation}
In addition, suppose $j_1\sim j_2$ with $j_1< j_2$. Then, again using \eqref{orthogonal} and the orthonormality of $\varphi_k$, as well as Cauchy-Schwarz,
\begin{equation}\label{nearorthogonal}|\sum_{n\sim j_1}a_{n,j_1}a_{n,j_2}|=|-\sum_{n\nsim j_1}a_{n,j_1}a_{n,j_2}|\leq(\sum_{n\nsim j_1}a_{n,j_1}^2)^{1/2}(\sum_{n\nsim j_1}a_{n,j_2}^2)^{1/2}\leq Ce^{-\tau j_1}.\end{equation}

Thinking of $A=\{a_{n,j}\}$ as an infinite matrix, let $M_i$ be the square submatrix of $A$ with $n$ and $j$ in $\mathcal L_i$; note that by condition (3) of Lemma \ref{clustering} $M_i$ is of size at most $2k\times 2k$. We interpret \eqref{unitvectors} and \eqref{nearorthogonal} as saying that the columns of $M_i$ have almost unit length and are almost orthogonal to each other. In particular, for each $i$, we may write $M_i^T M_i=I+R_i$. By \eqref{unitvectors} and \eqref{nearorthogonal}, there exists a constant $\tau>0$ such that $||R_i||_{\infty}\leq Ce^{-\tau i}$ (the norm $||M||_{\infty}$ denotes the supremum of the entries of $M$). Therefore, for large enough $i$, $(I+R_i)$ is always invertible, with inverse of the form $I+S_i$, with $||S_i||_{\infty}\leq Ce^{-\tau i}$ as well. In fact, $(I+R_i)^{-1}M_i^TM_i=I$, so $M_i^{-1}$ exists and equals $(I+R_i)^{-1}M_i^T=M_i^T+S_iM_i^T$. Multiply the equation $M_i^TM_i=I+R_i$ on the left by $M_i$ and on the right by $M_i^{-1}$ to yield
\[M_iM_i^T=M_i(I+R_i)M_i^{-1}=I+M_iR_i(I+S_i)M_i^T.\]
We see now that $M_iM_i^T$ has the form $I+Q_i$, where $Q_i$ satisfies $||Q_i||_{\infty}\leq Ce^{-\tau i}$ (note that the entries of $M_i$ and $M_i^T$ are bounded by 1). 

This now tells us that the \emph{rows} of $M_i$ have almost unit length: i.e. that
\[1-Ce^{-\tau i}\leq\sum_{j\sim n,\, j\in\mathcal L_i}a_{n,j}^2\leq 1.\]
By subtraction and \eqref{comparedrates}, there exists $\tau>0$ such that
\begin{equation}\label{transposebound}\sum_{j\nsim n}a^2_{n,j}\leq Ce^{-\tau n},\end{equation}
which is the analogue of \eqref{betterbound}, but summing in $j$ instead of in $n$. Additionally, it is an immediate consequence of \eqref{betterbound} and \eqref{transposebound} that if $j\nsim n$,
\begin{equation}\label{indivtermbounds}a_{n,j}^2\leq Ce^{-\tau j};\quad a_{n,j}^2\leq Ce^{-\tau n}.\end{equation}

By the Sobolev embedding theorem, for each $k\in\mathbb N_0$, we have $||\cdot||_{C^k(\partial D)}\leq C||\cdot||_{H^{k+1}(\partial D)}$. We apply this to $f_{\lambda_n}$, knowing that we can compute Sobolev norms of the $\bar e_j$ directly:
\[||f_{\lambda_n}||^2_{C^k}\leq C||f_n||^2_{H^{k+1}}\leq C\sum_{j\nsim n}(1+\mu_j)^{k+1}|a_{n,j}|^2.\]
By Weyl asymptotics of the $\mu_j$,
\[||f_{\lambda_n}||^2_{C^k}\leq C\sum_{j\nsim n}(1+j)^{k+1}a_{n,j}^2=C\sum_{j\nsim n,\, j<n}(1+j)^{k+1}a_{n,j}^2+C\sum_{j\nsim n,\, j>n}(1+j)^{k+1}a_{n,j}^2.\]
Using \eqref{indivtermbounds}, then the Weyl asymptotics again, we have
\begin{multline}||f_{\lambda_n}||^2_{C^k}\leq C\sum_{j\nsim n,\, j<n}(1+n)^{k+1}e^{-\tau n}+C\sum_{j\nsim n,\,j>n}(1+j)^{k+1}e^{-\tau j}\\
\leq Cn(1+n)^{k+1}e^{-\tau n}+C\int_{n}^{\infty}(1+t)^{k+1}e^{-\tau t}\,dt\leq C\lambda_n^{k+2}e^{-\tau n}.
\end{multline} Choosing $\tau$ slightly less than the current $\tau/2$, we can absorb $\lambda_n^2$ and take square roots, completing the proof.
\end{proof}

\begin{proof}[Proof of Proposition \ref{cor:eigenvaluedecay}] From the discussion before \eqref{L2bound}, considering each term individually, we see that for every $n$ and $j$,
\begin{equation}\label{easybound}|a_{n,j}(\lambda_n-\mu_j)|\leq Ce^{-\tau j}.\end{equation}
Using the matrix notation from the previous proof, let $V_i$ and $W_i$ be the diagonal matrices whose entries are $\{\lambda_j: j\in\mathcal L_i\}$ and $\{\mu_j: j\in\mathcal L_i\}$ respectively. Using \eqref{comparedrates}, \eqref{easybound} implies the statement
\[||M_iV_i-W_iM_i||_{\infty}\leq Ce^{-\tau i}.\]
For sufficiently large $i$, $M_i$ is invertible with inverse uniformly bounded in the $||\cdot||_{\infty}$ matrix norm. We deduce that
\[M_iV_iM_i^{-1}=W_i+\bar R_i,\ ||\bar R_i||_{\infty}\leq Ce^{-\tau i}.\]
So $W_i+\bar R_i$ is diagonalized by $M_i$ and has eigenvalues $\{\lambda_j:j\in\mathcal L_i\}$. By the Bauer-Fike theorem (see, e.g., \cite[Observation 6.3.1]{HJ}), the eigenvalues of $W_i=(W_i+\bar R_i)-\bar R_i$ lie in disks centered at each $\lambda_j$ of radius $||M_i||_{\infty}\cdot ||M_i||^{-1}_{\infty}\cdot ||R_i||_{\infty}\leq Ce^{-\tau i}$. Although the perturbed eigenvalues may move from one disk to another if there is overlap, there are at most $k$ disks, so they can move at most $2kCe^{-\tau i}$. Relabeling $C=2kC$, this shows that $|\lambda_n-\mu_n|\leq Ce^{-\tau i}$, where $\lambda_n\in\mathcal L_i$. The result follows immediately by another use of \eqref{comparedrates}, replacing $\tau$ with $\tau_1=\tau/C$.
\end{proof}

\end{document}